\documentclass[10pt]{article}
\pdfoutput=1
\usepackage{xcolor}
\usepackage{amsmath}
\usepackage{amssymb}
\usepackage{amsthm}
\usepackage{dsfont}
\usepackage{fullpage}
\usepackage{graphics}
\usepackage{graphicx}
\usepackage{subcaption}
\usepackage{caption}
\usepackage[font=footnotesize,labelfont=bf]{caption}
\usepackage{mathtools}
\usepackage{mathrsfs}
\usepackage{bbm}
\usepackage{multicol}
\usepackage[shortlabels]{enumitem}
\usepackage{pgfplots}
\pgfplotsset{compat=newest}
\usepackage{bm}
\usepackage{tikz}
\usepackage{comment}
\usetikzlibrary{patterns}
\usetikzlibrary{decorations.pathmorphing}
\usetikzlibrary{decorations.pathreplacing,decorations.markings}
\usetikzlibrary{shapes}
\usetikzlibrary{3d}
\usetikzlibrary{arrows.meta}
\usepackage{xcolor}
\usepackage{xifthen}
\usetikzlibrary{decorations}

\newtheorem{thm}{Theorem}
\newtheorem{lm}[thm]{Lemma}
\newtheorem{defn}[thm]{Definition}
\newtheorem{prop}[thm]{Proposition}
\newtheorem{rmk}[thm]{Remark}
\newtheorem{cor}[thm]{Corollary}

\newcommand{\ddbar}[2]{\frac{{\mathrm d}#1}{2\pi {\mathrm i}#2}}
\newcommand{\LL}{\mathrm{L}}
\newcommand{\RR}{\mathrm{R}}
\newcommand{\inn}{-}
\newcommand{\out}{+}
\newcommand{\complexC}{\mathbb{C}}
\newcommand{\intZ}{\mathbb{Z}}
\newcommand{\prob}{\mathbb{P}}
\newcommand{\realR}{\mathbb{R}}
\newcommand{\rC}{\mathrm{C}}
\newcommand{\rd}{\mathrm{d}}
\newcommand{\rI}{\mathrm{I}}
\newcommand{\rx}{\mathrm{x}}
\newcommand{\rf}{\mathrm{f}}
\newcommand{\ri}{\mathrm{i}}
\newcommand{\ch}{\mathrm{ch}}
\newcommand{\bn}{\mathbf{n}}
\newcommand{\1}{\mathbf{1}}
\newcommand{\rT}{\mathrm{T}}
\newcommand{\rS}{\mathrm{S}}
\newcommand{\hu}{\hat{u}}
\newcommand{\rF}{\mathrm{F}}
\newcommand{\E}{\mathbb{E}}
\newcommand{\rX}{\mathrm{X}}
\newcommand{\ftn}{F}
\newcommand{\lbd}{s}
\newcommand{\ubd}{t}
\newcommand{\cc}{p}
\newcommand{\ELPP}{\mathrm{L}}

\newcommand{\annu}{\Omega^{\circledcirc}}
\newcommand{\bbeta}{B}
\newcommand{\qq}{\phi}

\numberwithin{equation}{section} 
\numberwithin{thm}{section}
	
\author{
		Zhipeng Liu\footnote{
        Email: \texttt{zhipengliumath@gmail.com}}
        \and
        Tejaswi Tripathi\footnote{Department of Mathematics, University of Kansas, Lawrence, KS 66045. Email: \texttt{tejaswit@ku.edu}}}
	\title{A determinant identity for the sum of contour integral matrices}
	
\begin{document}
		\maketitle
        \begin{abstract}
        We derive an identity for the determinant of the sum of two $n\times n$ matrices, $U$ and $M$, whose entries are defined via contour integrals. Specifically, we consider $U(i,j)=\frac{1}{2\pi\ri}\oint_\rC \frac{\prod_{\ell=1}^{i-1} (z-\beta_\ell)}{\prod_{\ell=1}^{j} (z-\beta_\ell)} p_i(z)f_j(z)\rd z$ and $M(i,j)= \frac{1}{2\pi\ri}\int_{\Gamma} q_i(z)g_j(z) \rd z$. Under suitable assumptions on the functions $p,q,f,g$, we show that $\det(U+M)$ can be expressed as a Fredholm determinant $\det(\rI +K)$, where $K$ is an integral kernel acting on the contour $\Gamma$.       The kernel $K$ depends on a function $H$ that solves a system of integral equations. When $f_i$ and $g_i$ are specialized to certain rational functions depending on two sets of parameters $(\alpha_\ell)_{\ell\in\intZ}$ and $(\beta_\ell)_{\ell\in\intZ}$, $H$ becomes the characteristic function associated with inhomogeneous directed last passage percolation (DLPP) and inhomogeneous totally asymmetric simple exclusion process (TASEP) models. Furthermore, we obtain an explicit random walk hitting expectation representation of this characteristic function.

        Our work generalizes a recent identity by Baik, Liao, and Liu (2026), which plays an important role in finding the multipoint distribution formula of the periodic KPZ fixed point. Finally, we demonstrate three applications of our general formulas in integrable probability: a new Fredholm determinant formula for the distribution of the path-to-point last passage time in the inhomogeneous DLPP, an indirect proof of a new path-to-line joint distribution formula in the homogeneous DLPP, and a novel proof of the TASEP path-integral formula previously obtained by Matetski, Quastel, and Remenik (2021). 
        \end{abstract}

\section{Introduction}

Very recently, Baik, Liao, and Liu \cite{Baik-Liao-Liu26} derivevd a determinant identity that plays an important role in analyzing the energy function of the periodic totally asymmetric simple exclusion process with a general initial condition. The identity is as follows, modulo a simple conjugation.
\begin{prop}[\cite{Baik-Liao-Liu26}]
\label{prop:BLL}
    Let $f_1,\ldots,f_n$ be analytic functions on a neighborhood $\Omega_0$ of $0$ satisfying $f_i(0)=1$ for all $i$. Let $\Gamma$ be a simple contour contained in the region $\{u\in\complexC: |u+1|<1\}\setminus\{u: |u|<\epsilon\}$ for some positive constant $\epsilon$, and let $g_1,\ldots,g_n$ be bounded continuous functions on $\Gamma$. Suppose there exists a function $H(v,u)$ on $\Omega_0\times \Gamma$ that is analytic at $v=0$ and satisfies 
    \begin{equation}
        \oint_0 \frac{f_i(v)}{v^i} H(v,u) \ddbar{v}{} = g_i(u),\quad \text{for all }1\le i\le n.
    \end{equation}
    Then,
    \begin{equation}
        \det\left[\oint_0 v^{i-j-1}f_j(v) \ddbar{v}{} +\int_{\Gamma} u^{i-1}g_j(u) \ddbar{u}{}\right]_{i,j=1}^n =\det (\rI+K),
    \end{equation}
    where $K$ is an integral kernel acting on $L^2(\Gamma,\rd u/2\pi\ri)$ given by
    \begin{equation}
    \label{eq:specialK}
        K(u_1,u_2) = \oint_0 \frac{u_2^n H(v,u_1)}{v^n(u_2-v)}\ddbar{v}{}.
    \end{equation}
\end{prop}

The main goal of this paper is to provide a generalization of the identity above. The main result is as follows. 

\begin{thm}
\label{prop:main}
   Let $n$ be a positive integer. Suppose that $\Omega$ is a simply connected subset of the complex plane containing the points $\beta_1, \ldots, \beta_n$, and $\rC$ is a simple closed contour in $\Omega$ that encloses the points $\beta_1, \ldots, \beta_n$. Let $\Gamma$ be an arbitrary simple contour in the complex plane. Assume that for each $1\le i\le n$, $p_i$ and $f_i$ are analytic functions in $\Omega$ satisfying $p_i(\beta_i)=f_i(\beta_i)=1$, and $p_i(v)\ne 0$ for all $v\in\Omega$. Let $q_i$ and $g_i$ be square integrable continuous functions on $(\Gamma,|\rd u|)$, for all $i$, i.e., $\int_\Gamma |q_i(u)|^2 |\rd u|<\infty$ and $\int_\Gamma |g_i(u)|^2 |\rd u|<\infty$ for all $i$. Suppose there exists a function $H(v,u)$ on $\Omega \times \Gamma$ that is analytic in $v$ and satisfies 
    \begin{equation}
    \label{eq:fgHrel}
        \oint_\rC \frac{f_i(v)}{\prod_{\ell=1}^{i} (v-\beta_\ell)} H(v,u) \ddbar{v}{} = g_i(u),\quad \text{for all }1\le i\le n.
    \end{equation}
     Then we have
    \begin{equation}
    \label{eq:findettoFred}
        \det\left[ \oint_\rC \frac{\prod_{\ell=1}^{i-1} (v-\beta_\ell)}{\prod_{\ell=1}^{j} (v-\beta_\ell)} p_i(v)f_j(v)\ddbar{v}{} + \int_{\Gamma}q_i(u)g_j(u)\ddbar{u}{} \right]_{i,j=1}^n =\det(\rI+K),
    \end{equation}
     where $K$ is an integral kernel acting on $L^2(\Gamma,\rd u/2\pi\ri)$ given by
    \begin{equation}
    \label{eq:def_Kernel_K}
        K(u_1,u_2) = \sum_{\ell=1}^n q_\ell(u_2) \oint_{\rC_1} \ddbar{v_1}{}\cdots \oint_{\rC_\ell}\ddbar{v_\ell}{} \frac{H(v_1,u_1)}{\left[\prod_{i=1}^{\ell-1}(v_i-v_{i+1})\right]\cdot \left[\prod_{i=1}^\ell (v_i-\beta_i)P_i(v_i)\right]},
    \end{equation}
    where $P_1(v):=p_1(v)$, and $P_i(v):=p_i(v)/p_{i-1}(v)$ for $2\le i\le n$. $\rC_1,\ldots,\rC_n$ are nested simple closed contours in $\Omega$ enclosing the points $\beta_1,\ldots,\beta_n$, ordered from outermost to innermost.
\end{thm}
\begin{rmk}
    The conditions $p_i(\beta_i)=f_i(\beta_i)=1$ can be relaxed to $p_i(\beta_i)\ne 0$ and $f_i(\beta_i)\ne 0$. Under these relaxed conditions, applying Theorem \ref{prop:main} to the normalized functions $\hat p_i(v)=p_i(v)/p_i(\beta_i)$, $\hat f_j(v)=f_j(v)/f_j(\beta_j)$, $\hat q_i(u)=q_i(u)/p_i(\beta_i)$ and $\hat g_j(u)=g_j(u)/f_j(\beta_j)$  yields
    \begin{equation}
        \det\left[ \oint_\rC \frac{\prod_{\ell=1}^{i-1} (v-\beta_\ell)}{\prod_{\ell=1}^{j} (v-\beta_\ell)} p_i(v)f_j(v)\ddbar{v}{} + \int_{\Gamma}q_i(u)g_j(u)\ddbar{u}{} \right]_{i,j=1}^n =\prod_{i=1}^n p_i(\beta_i)f_i(\beta_i)\cdot \det(\rI+K),
    \end{equation}
    with $K$ defined in the same way as in \eqref{eq:def_Kernel_K}.    
\end{rmk}

\begin{rmk}
    The initial draft of this paper was restricted to the homogeneous case ($\beta_i=0$ for $i=1,\ldots,n$). Following the appearance of this identity, Leo Petrov recognized that  the left-hand side of the homogeneous version of \eqref{eq:findettoFred} can be viewed as a tilted Toeplitz minor. This observation led to the derivation of a distinct Fredholm determinant formula in \cite{Petrov2026}, which generalizes the well-known Borodin-Okounkov-Geronimo-Case identity \cite{Geronimo-Case79,Borodin-Okounkov00}. 
\end{rmk}

We further remark that we could replace $\Gamma$ by a general set and $\rd u$ by $\rd \mu(u)$ where $\mu$ is a measure on the set. More explicitly, we have 
\begin{prop}
    Let $\mu$ be a complex measure on a measurable space $(S,\Sigma)$. Theorem \ref{prop:main} holds if we replace $(\Gamma,\rd u)$ by $(S, \rd \mu(u))$.
\end{prop}
The proof is the same as that for Theorem \ref{prop:main} since it does not rely on the explicit space or measure.

\begin{rmk}
    Our identity has the following connection to the well-known identity $\det (\rI +K_1K_2) =\det(\rI+K_2K_1)$. In fact, if $K_1$ and $K_2$ are two operators with kernels $K_1(i,u)=q_i(u)$ for $(i,u)\in \{1,\ldots,n\}\times S$, and $K_2(u,j)=g_j(u)$ for $(u,j)\in S\times \{1,\ldots,n\}$ respectively, then the identity $\det (\rI +K_1K_2) =\det(\rI+K_2K_1)$ says that
    \begin{equation}
        \det\left[ \delta_i(j)  + \int_{S} q_i(u)g_j(u)\rd \mu(u)\right]_{i,j=1}^n = \det (\rI + K_2K_1)
    \end{equation}
    where $K_2K_1$ is an integral kernel acting on $L^2(S,\rd\mu)$ given by
    \begin{equation}
    \label{eq:kernel_BA}
        (K_2K_1)(u_1,u_2) = \sum_{\ell=1}^n q_\ell(u_2)g_\ell(u_1).
    \end{equation}

    The identity \eqref{eq:findettoFred} can be viewed as
    \begin{equation}
        \det(U +K_1K_2) = \det(U)\cdot \det(\rI + K),
    \end{equation}
    where $U$ is an upper triangular matrix, and $K=K_2U^{-1}K_1$. It is generally hard to find the explicit formula for $U^{-1}$ or $K$. The main novelty of our identity is the explicit expression of the kernel $K$ in \eqref{eq:def_Kernel_K} when $U(i,j)=\oint_{\rC} \frac{\prod_{\ell=1}^{i-1} (v-\beta_\ell)}{\prod_{\ell=1}^{j} (v-\beta_\ell)}p_i(v)f_j(v)\frac{\rd v}{2\pi\ri}$.
\end{rmk}

\bigskip

Throughout this paper, we use the following general product notation,
\begin{equation}
    \prod_{i=\lbd}^{\ubd}a_i = \frac{\prod_{i=-L}^\ubd a_i}{\prod_{i=-L}^{\lbd-1}a_i},\quad \text{for all } \lbd,\ubd\in\intZ,
\end{equation}
where $L$ is an integer satisfying  $-L\le \ubd$ and $-L\le \lbd-1$.  Using this notation, we have the following identity
\begin{equation}
\label{eq:product_identity}
    \left(\prod_{i=\lbd}^\ubd a_i \right) \cdot \left(\prod_{i=\ubd+1}^{\lbd -1}a_i\right)=1.
\end{equation}

The following lemma will be used frequently in this paper.

\begin{lm}
\label{lm:telescope}
    Suppose $(a_i)_{i\in\intZ}$ is a sequence  of points, and $z,w$ are two distinct complex numbers in $\complexC\setminus\{a_i:i\in\intZ\}$. Let $\lbd,\ubd,A$, and $B$ be four integers. Assume that $A\le B$. We have 
    \begin{equation}
        \sum_{x=A}^B \frac{\prod_{i=\lbd}^x(z-a_i)}{\prod_{i=\lbd}^{x+1}(w-a_i)} = \frac{1}{z-w} \left[ \prod_{i=\lbd}^{B+1}\frac{ z-a_i}{w-a_i}-\prod_{i=\lbd}^{A}\frac{ z-a_i}{w-a_i}\right],
    \end{equation}
    and
    \begin{equation}
        \sum_{x=A}^B \frac{\prod_{i=x+1}^\ubd(z-a_i)}{\prod_{i=x}^{\ubd}(w-a_i)} = \frac{1}{z-w} \left[ \prod_{i=A}^{\ubd}\frac{ z-a_i}{w-a_i}-\prod_{i=B+1}^{\ubd}\frac{ z-a_i}{w-a_i}\right].
    \end{equation}
\end{lm}
\begin{proof}
It is easy to verify both identities when $A=B$. The general identity follows from summing the telescoping sequences. 
\end{proof}

\subsection{Two specializations}
\label{sec:specializations}

In this subsection, we discuss two specializations of Theorem \ref{prop:main}, which are detailed in the following sections. We first show that when $P_i(v)=p_i(v)/p_{i-1}(v)$ is a constant function for $i\ge 2$, the kernel $K$ simplifies to a single contour integral. See the formula \eqref{eq:simplified_kernel}. Then for the second specialization, where $f_i$ and $g_i$ are certain rational functions depending on two sets of parameters $(\alpha_i)_{i\in\intZ},(\beta_i)_{i\in\intZ}$, we identify a function $H$ that satisfies the desired properties in Theorem \ref{prop:main} in terms of an explicit random walk hitting expectation. See Theorem \ref{thm:hitting}.

\subsubsection{Simplification of the kernel when $P_i=\,$constant for all $i\ge 2$}
\label{sec:check}

When $P_i(v)$ is a nonzero constant function for each $i\ge 2$, the $v_i$ integrals in the kernel  \eqref{eq:def_Kernel_K} can be simplified
\begin{equation}
\begin{split}
    \oint_{\rC_1} \ddbar{v_1}{}\cdots \oint_{\rC_\ell}\ddbar{v_\ell}{} \frac{H(v_1,u_1)}{\left[\prod_{i=1}^{\ell-1}(v_i-v_{i+1})\right]\cdot \left[\prod_{i=1}^\ell (v_i-\beta_i) P_i(v_i)\right]}
   &=\oint_{\rC} \ddbar{v_1}{} \frac{H(v_1,u_1)}{p_\ell(v_1)\prod_{i=1}^\ell (v_1-\beta_i)}
\end{split}
\end{equation}
by evaluating the $v_\ell,\ldots,v_2$ integrals consecutively (and deforming the contours to infinity). Then the kernel becomes
\begin{equation}
\label{eq:simplified_kernel}
    K(u_1,u_2) = \oint_\rC \ddbar{v}{} \left[\sum_{\ell=1}^n \frac{q_\ell(u_2)/p_\ell(v)}{\prod_{i=1}^\ell (v-\beta_i)}\right]\cdot H(v,u_1).
\end{equation}

The summation in the above integrand can be further simplified when
\begin{equation}
q_\ell(u)/p_\ell(v)=  Q(u,v)\cdot \prod_{i=1}^{\ell-1} (u-\beta_i),\quad \text{for all }1\le \ell \le n,
\end{equation}
where $Q(u,v)$ is a function independent of $\ell$. Note that $Q(u,v)$ has to be analytic in $v\in\Omega$ since $p_\ell$ is nonzero and analytic in $\Omega $. In such a case, we can use Lemma \ref{lm:telescope} to obtain 
\begin{equation}
\label{eq:simplified_kernel2}
    K(u_1,u_2) = \oint_{\rC} \frac{\rd v}{2\pi\ri} Q(u_2,v )\cdot \left[\prod_{\ell=1}^n \frac{u_2-\beta_\ell}{v-\beta_\ell}-1\right]\cdot \frac{H(v,u_1)}{u_2-v} = \oint_{\rC} \frac{\rd v}{2\pi\ri} \prod_{\ell=1}^n \frac{u_2-\beta_\ell}{v-\beta_\ell} \cdot \frac{Q(u_2,v )H(v,u_1)}{u_2-v}, 
\end{equation}
where we used the fact that $Q(u_2,v)\cdot H(v,u_1)/(v-u_2)$ is analytic for $v\in\Omega$ in the last equation.

In particular, under the assumptions in Proposition \ref{prop:BLL}, we have $p_i(v)\equiv 1$, $\beta_i=0$ for all $i$, and $q_\ell(u)=u^{\ell-1}$. This kernel matches \eqref{eq:specialK}.

\subsubsection{Random walk hitting expectation representation for the inhomogeneous characteristic function}

One important specialization of the function $H$ is the following inhomogeneous characteristic function which arises in the study of inhomogeneous directed last passage percolation (DLPP) \cite{Hong-Liu-Tripathi26} and equivalently, the  totally asymmetric simple exclusion process (TASEP) with inhomogeneous jumping rates. This function carries the information of the boundary condition in the inhomogeneous DLPP, or the initial condition in the inhomogeneous TASEP.

  \begin{defn}
  \label{def:inhomogeneous_ch}
      Let $\boldsymbol{\lambda}=(\lambda_1, \ldots,\lambda_n)\in\intZ^n$ satisfy $\lambda_1\ge  \cdots\ge \lambda_n$. Let $\Omega_\LL$ and $\Omega_\RR$ be two  simply connected and bounded domains on $\complexC$ such that $\mathrm{dist}(\Omega_\LL,\Omega_\RR):=\inf\{|u-v|: u\in\Omega_\LL, v\in\Omega_\RR\}>0$. Assume that $(\alpha_j)_{j\in\intZ}$ and $(\beta_j)_{j\in\intZ}$ are two sequences of points in $\Omega_\LL$ and $\Omega_\RR$ respectively, and the sequence $(\alpha_i)_{i\in\intZ}$ satisfies 
      \begin{equation}
      \label{eq:assumption_beta}
          \left|\frac{u-\alpha_i}{v-\alpha_i}\right|<1-\epsilon,\quad \text{ for all }u\in\Omega_\LL, v\in\Omega_\RR, \text{ and }i\in\intZ.
      \end{equation}
      Here $\epsilon>0$ is a fixed constant.   See Figure \ref{fig:contour_domains} for an illustration.   Suppose $\lbd$ is any fixed integer. 
      We call $\ch =\ch_{\boldsymbol{\lambda}} :\Omega_\RR\times\Omega_\LL\to\complexC$ an \emph{inhomogeneous characteristic function} of $\boldsymbol{\lambda}$ if it satisfies the following two conditions:
      \begin{enumerate}[(a)]
          \item $\ch(v,u)$ is an analytic function of $v$ in $\Omega_\RR$ for any fixed $u\in\Omega_\LL$.
          \item For each $m=1,\ldots,n$, the following integral equation holds
          \begin{equation}
          \label{eq:char_linear_equations}
              \oint_\rC \ftn_m(v) \ch(v,u)\frac{\rd v}{2\pi\ri} = -\ftn_m(u),
          \end{equation}
          where $\rC$ is a positively oriented contour in $\Omega_\RR$ that encloses all the points $\beta_1,\ldots,\beta_n$, and the functions $\ftn_m$ are defined by
          \begin{equation}
              \label{eq:def_ftn}
              \ftn_m(w):= \frac{\prod_{i=\lbd}^{\lambda_m}(w-\alpha_i)}{\prod_{i=1}^m(w-\beta_i)},\quad m=1,\ldots,n.
          \end{equation}
      \end{enumerate}
  \end{defn}

\begin{figure}
\centering
\begin{tikzpicture}[
    % Style for the directed integration contour
    directed/.style={postaction={decorate}, decoration={markings, mark=at position 0.65 with {\arrow{Latex[length=2mm]}}}},
    dot/.style={circle, fill=black, inner sep=1pt}
]

    % 1. Draw the complex plane axes
    \draw[->, thick, gray] (-7, 0) -- (2.5, 0) node[right, text=black] {$\mathrm{Re}$};
    \draw[->, thick, gray] (0, -2.5) -- (0, 2.5) node[above, text=black] {$\mathrm{Im}$};

    % Mark 0 and -1 reference centers
    \coordinate (Origin) at (0,0);
    \coordinate (MinusOne) at (-5,0); % Scaled to x=-3 for clear visual separation
    
    %\node[below right] at (Origin) {$0$};
    %\node[below right] at (MinusOne) {$-1$};

    % 2. Draw the bounded, simply connected domains
    % Domain \Omega_L around -1
    %\fill[blue!5] (MinusOne) circle (1.2);
    \draw[blue, thick] (MinusOne) circle (1.2);
    \node[blue, above left] at (-5.8, 0.8) {$\Omega_\LL$};

    % Domain \Omega_R around 0
    %\fill[red!5] (Origin) circle (1.4);
    \draw[red, thick] (Origin) circle (1.4);
    \node[red, above right] at (1.0, 1.0) {$\Omega_\RR$};

    % 3. Draw contour C inside \Omega_R enclosing the \beta points
    \draw[directed, thick, black] (Origin) circle (0.8);
    \node[black, right] at (0.6, 0.6) {$\rC$};

    % 4. Plot the variables
    % \alpha sequences in \Omega_L
    \node[dot, label={[scale=0.8]above:$\alpha_1$}] at (-4.8, 0.4) {};
    \node[dot, label={[scale=0.8]left:$\alpha_2$}] at (-5.4, 0.1) {};
    \node[dot, label={[scale=0.8]below:$\alpha_3$}] at (-5.1, -0.4) {};
    
    % Arbitrary point u in \Omega_L
    \node[dot, blue, label={[blue, scale=0.9]right:$u$}] at (-4.4, 0.5) {}; 

    % \beta sequences in \Omega_R (strictly inside contour C)
    \node[dot, label={[scale=0.8]above:$\beta_1$}] at (0.2, 0.3) {};
    \node[dot, label={[scale=0.8]left:$\beta_2$}] at (-0.3, 0.1) {};
    \node[dot, label={[scale=0.8]below:$\beta_n$}] at (0.1, -0.3) {};
    
    % Integration variable v moving along contour C
    \node[dot, red, label={[red, scale=0.9]right:$v$}] at (0.48, -0.64) {};
\end{tikzpicture}

\caption{Illustration of the domains $\Omega_\LL,\Omega_\RR$, points $\beta_i,\alpha_i$, and the contour $\rC$}
\label{fig:contour_domains}
\end{figure}

Note that we have one more set of parameters and one more domain in this definition compared to the original $H$ function appeared in Theorem \ref{prop:main}. The domain $\Omega_\RR$ in this definition corresponds to $\Omega$ in Theorem \ref{prop:main}. Moreover, if $f_i(v)=\ftn_i(v)\cdot \prod_{\ell=1}^i(v-\beta_i)$ and $g_i(u) = c\cdot \ftn_i(u)$, where $c$ is a fixed constant, the $H$ function appearing in Theorem \ref{prop:main} can be expressed as $H(v,u)=-c\cdot \ch(v,u)$.

  We remark that the homogeneous version of this characteristic function, when $\alpha_i=-1$ and $\beta_i=0$ for all $i$, and when $\lbd=0$, has appeared in the DLPP with homogeneous weights and the homogeneous TASEP \cite{Liu22a,Liao-Liu25,Baik-Liao-Liu26}, and it is called the characteristic function of the initial condition $Y=(y_i=\lambda_i-i+1)_{i\ge 1}$ for TASEP.

\bigskip

The main goal of this subsection is to provide an inhomogeneous characteristic function which can be explicitly expressed as a random walk hitting expectation. A similar expression for the homogeneous case has been obtained in \cite{Liao-Liu25,Baik-Liao-Liu26}, which plays an important role in the asymptotic analysis of the TASEP and the periodic TASEP with general initial conditions. We expect our expression for the inhomogeneous characteristic function to play an analogous role in the inhomogeneous setting.

  The explicit formula of this inhomogeneous characteristic function involves a discrete Markov chain $(G_k)_{k\ge 0}$ with complex-valued transition probabilities defined as follows 
  \begin{equation}
      \label{eq:transition_prob}
      \prob(G_{k}=x\mid G_{k-1}=y) = \begin{dcases}
          (\beta_{k}-\cc)\frac{\prod_{i=x+1}^y(\cc-\alpha_i)}{\prod_{i=x}^y (\beta_k-\alpha_i)}, & x\le y,\\
          0,& x>y.
      \end{dcases}
  \end{equation}
  Here $\cc$ is a parameter satisfying
  \begin{equation}
  \label{eq:cc_assumption}
      \left|\frac{\cc-\alpha_i}{\beta_j-\alpha_i}\right|<1-\epsilon,\quad \text{for all}\quad i,j,
  \end{equation}
  with some positive constant $\epsilon$. Here we slightly abuse the notation $\epsilon$, which appears in both \eqref{eq:assumption_beta} and \eqref{eq:cc_assumption} but they may denote two different constants.
   We will verify the following normalization identity in Section \ref{sec:verification_transition}, hence the formula \eqref{eq:transition_prob} gives a (complex-valued) probability measure.
  \begin{equation}
  \label{eq:sum_normalized}
      \sum_{x\in\intZ}\prob(G_{ k}=x\mid G_{k-1}=y)=1,\quad y\in\intZ.
  \end{equation}
  
  We define the stopping time $\tau:=\min\{m\ge 0: G_m>\lambda_{m+1}\}$.  
  \begin{thm}
  \label{thm:hitting}
      The following formula is well-defined and is an inhomogeneous characteristic function of $\boldsymbol{\lambda}$:
      \begin{equation}
          \label{eq:ch_hitting_exp}
          \ch(v,u) = \sum_{x\in\intZ} \prod_{i=\lbd}^{x-1}(u-\alpha_i)\cdot \E_{G_0=x}\left[ \frac{\prod_{j=1}^\tau (v-\beta_j)}{\prod_{i=\lbd}^{G_\tau}(v-\alpha_i)}\frac{1}{\prod_{i=1}^\tau(\cc-\beta_i) \cdot \prod_{j=G_\tau+1}^x (\cc-\alpha_j)}\1_{\tau<n}\right].
      \end{equation}
  \end{thm}

\begin{rmk}
    Random walk hitting expectations have emerged in the correlation kernel of equal-time multipoint distributions for TASEP and the KPZ fixed point \cite{Matetski-Quastel-Remenik21}. Subsequent generalizations \cite{Matetski-Remenik23,Bisi-Liao-Saenz-Zygouras23,Matetski-Remenik25} have also featured random walk hitting expectations in their correlation kernels. In particular, the studies \cite{Bisi-Liao-Saenz-Zygouras23,Matetski-Remenik25} investigated inhomogeneous TASEP, where the underlying random walks are geometric. It is unclear how to directly relate those geometric random walk representations to the random walk defined in the present work.

    From a different perspective, similar random walk hitting expectations have appeared in the characteristic functions of both standard and periodic TASEP \cite{Liao-Liu25,Baik-Liao-Liu26} with homogeneous rates, and are crucial for identifying the space-time joint distributions of these two models. Theorem \ref{thm:hitting} can be viewed as a generalization of the characteristic function formula in \cite{Liao-Liu25}.
\end{rmk}
\begin{rmk}\label{rmk:step_IC}
    When $\lambda_1=\cdots=\lambda_n$, either $\tau=0$ or $\tau\ge n$ holds since the random walk always weakly decreases. $\tau=0$ if and only if $G_0=x>\lambda_1$. Thus we can simplify the summation in \eqref{eq:ch_hitting_exp} and obtain 
    \begin{equation}
        \ch(v,u) = \sum_{x=\lambda_1+1}^\infty \frac{\prod_{i=\lbd}^{x-1}(u-\alpha_i)}{\prod_{i=\lbd}^{x}(v-\alpha_i)} = \frac{1}{v-u} \prod_{i=\lbd}^{\lambda_1}\frac{u-\alpha_i}{v-\alpha_i},
    \end{equation}
    where we used Lemma \ref{lm:telescope} and the assumption \eqref{eq:assumption_beta} in the last step.
\end{rmk}

  For the homogeneous case  $\alpha_i=-1$ and $\beta_i=0$,  $G_k$ is a geometric random walk with the transition probability $\prob(G_k=x\mid G_{k-1}=y)=(-\cc)\cdot (\cc+1)^{y-x}\1_{x\le y}$ with the parameter $\cc<0$. \eqref{eq:ch_hitting_exp} becomes 
  \begin{equation}
  \label{eq:homogeneous_characteristic_function}
      \ch(v,u) = \sum_{x\in\intZ} \left(\frac{u+1}{\cc+1}\right)^{x-\lbd} \E_{G_0=x} \left[ \frac{1}{\cc+1}\left(\frac{\cc+1}{v+1}\right)^{G_\tau-\lbd+1} \left(\frac{v}{\cc}\right)^\tau\1_{\tau<n}\right].
  \end{equation}
  If we set $y_i=\lambda_i-i+1$, $\lbd=0$, and $\cc=-\rho$, this function matches the characteristic function associated with $Y=(y_i)_{i\ge 1}$ in \cite[Definition 5.5]{Baik-Liao-Liu26} (and \cite[Theorem 3.4]{Liao-Liu25} if we further take $\cc=-1/2$) modulo a shift in the definition of the random walk $\hat G_k=G_k-k$. 

\medskip

 In Section \ref{sec:well-definedness-hitting}, we will show that the summation on the right-hand side of \eqref{eq:ch_hitting_exp} is absolutely convergent, and hence it is well-defined. The proof of Theorem \ref{thm:hitting} is given in Section \ref{sec:proof-hitting}.

\subsection{Applications}

We expect that Theorem \ref{prop:main} will have many applications in integrable probability, particularly for models with determinantal structures. Indeed, as shown in \cite{Baik-Liao-Liu26}, a special case of this result (Proposition \ref{prop:BLL}) already plays an important role in determining the multipoint distributions of the periodic KPZ fixed point with general initial conditions. In this subsection, we present Theorem \ref{thm:LPP-path-to-point} as an illustrative example to demonstrate one application of Theorem \ref{prop:main}. Two other applications are discussed in Section \ref{sec:applications}. 

Consider the inhomogeneous DLPP model. Assume $w_{i,j}$,  $(i,j)\in\intZ^2$, are independent exponential random variables with parameters  $-\alpha_i+\beta_j$, i.e.,
\begin{equation}
\begin{split}
    \prob( w_{i,j} \le x) = (1- e^{-(-\alpha_i+\beta_j)x})\cdot\1_{x\ge 0},\quad x\in\realR.
\end{split}
\end{equation}
Here we assume that  $\alpha_i<\beta_j$ for all $i,j\in\intZ$, so that the rates for these random variables are positive. Define the point-to-point last passage time from $\mathbf{p} \in\intZ^2$ to $\mathbf{q}\in \intZ^2$ as follows
\begin{equation}
    \ELPP(\mathbf{p};\mathbf{q}):= \max_{\pi} \sum_{\mathbf{r}\in\pi}  w_{\mathbf{r}},
\end{equation}
where the maximum is taken over all up/right lattice paths from $\mathbf{p}$ to $\mathbf{q}$. If such a path does not exist, we define $\ELPP(\mathbf{p};\mathbf{q})=-\infty$. We can also define the set-to-point last passage time from a set $S\subseteq \intZ^2$ to $\mathbf{q}\in\intZ^2$ by
\begin{equation}
    \ELPP (S;\mathbf{q}) = \max_{\mathbf{p}\in S}\ELPP (\mathbf{p};\mathbf{q}).
\end{equation}

Let $\boldsymbol{\lambda}=(\lambda_i)_{i\in\intZ_+}$ be a decreasing sequence of integers, $\lambda_1\ge \lambda_2 \ge\cdots$. For notation simplicity, we denote
\begin{equation}
    \ELPP_{\boldsymbol{\lambda}}(\mathbf{q}) = \ELPP( \{(\lambda_i+1,i): i\in\intZ_+\}; \mathbf{q}),
\end{equation}
and, for a fixed $n\ge 1$,
\begin{equation}
    \ELPP_{\boldsymbol{\lambda}}(m) = \left( \ELPP_{\boldsymbol{\lambda}}(m,1),\ldots, \ELPP_{\boldsymbol{\lambda}}(m,n)\right).
\end{equation}
Define $W_n:=\{(x_1,\ldots,x_n)\in\realR^n: x_n\ge \cdots\ge x_1\ge 0\}$. Then  $(\ELPP_{\boldsymbol{\lambda}}(m))_{m\ge \lambda_1+1}$ is a Markov process in the space $W_n$, and corresponds to the last passage times from a general initial path to a single column. The following formula for the density of $\ELPP_{\boldsymbol{\lambda}}(m)$ has been derived in \cite{Hong-Liu-Tripathi26}\footnote{An analogous formula for the inhomogeneous geometric DLPP was also derived in \cite{Hong-Liu-Tripathi26}.}.
\begin{thm}[\cite{Hong-Liu-Tripathi26}]
    \label{thm:LPP_path-to-line}
    For $m\ge \lambda_1+1$, we have the following determinantal formula for  $\ELPP_{\boldsymbol{\lambda}}(m)$,
    \begin{equation}
    \label{eq:LPP_path-to-line}
    \begin{split}
        \prob\left( \ELPP_{\boldsymbol{\lambda}}(m) \in \rd x\right) &=\det\left[ M_{i,j}\right]_{i,j=1}^n \rd x,\quad x=(x_1,\ldots,x_n)\in  W_n,
    \end{split}
    \end{equation}
    where
    \begin{equation}
    \begin{split}    
    M_{i,j} &= \oint_{|z|=R} \frac{\rd z}{2\pi\ri}  e^{x_i(z-\beta_i)}
    \frac{\prod_{k=1}^i (z-\beta_k)}{\prod_{k=1}^j (z-\beta_k)} \prod_{k=\lambda_j+1  }^m \frac{\beta_j-\alpha_k}{z-\alpha_k},
    \end{split}
\end{equation}
and $R$ is a large positive number such that the integration contour encloses all possible finite poles in the integrand.
\end{thm}

We remark that when $\lambda_i=0$ for all $i$, this formula is a special case of the column-to-column transition probability formula \cite[Corollary 3.1]{Johansson-Rahman22} (modulo the negative signs of the parameters $\alpha_i$). For general ${\boldsymbol{\lambda}}$, there is no other single determinantal formula for the joint distribution of the path-to-column last passage times, even in the homogeneous case. There are also Fredholm determinantal formulas in the language of TASEP (see \cite{Liu22a,Liao21}), given in terms of $(n-1)$-fold of contour integrals of a complicated Fredholm determinant. In fact, the original motivation to derive our main Theorem \ref{prop:main} in this paper comes from a different project \cite{Liu-Tripathi26}. In that work, we need a simple formula of the path-to-column last passage times to study the path-to-path geodesic, generalizing the approach used for point-to-point geodesics in \cite{Liu22b}. However, after deriving a homogeneous version of Theorem \ref{thm:LPP_path-to-line} using Theorem \ref{prop:main}, we discovered a straightforward proof and its inhomogeneous generalization. Since the direct proof is irrelevant to the topic of this paper, it will be provided in \cite{Hong-Liu-Tripathi26}.  Later in Section \ref{subsec:equivalence_formulas_column}, we will also provide the indirect derivation of the homogeneous degeneration of this formula using the Fredholm determinant formula in  \cite{Liu22a}, as an application of Theorem \ref{prop:main}.

Using Theorem \ref{thm:LPP_path-to-line} and Theorem \ref{prop:main}, it is straightforward to obtain the following result.
\begin{thm}
\label{thm:LPP-path-to-point}
Assume that $\boldsymbol{\lambda}=(\lambda_1, \ldots,\lambda_n)\in\intZ^n$ satisfies $\lambda_1\ge  \cdots\ge \lambda_n$, and $m\in\intZ$ satisfies $m\ge \lambda_1+1$. Let $\Omega_\LL$ and $\Omega_\RR$ be two  bounded, simply connected domains in $\complexC$ such that $\mathrm{dist}(\Omega_\LL,\Omega_\RR)>0$. Suppose $(\alpha_j)_{j\in\intZ}$ and $(\beta_j)_{j\in\intZ}$ are two sequences of points in $\Omega_\LL$ and $\Omega_\RR$ respectively, such that $\alpha_i<\beta_j$ for all $i,j\in\intZ$. Moreover, there exists some fixed constant $\epsilon>0$, such that
      \begin{equation}
          \left|\frac{u-\alpha_i}{v-\alpha_i}\right|<1-\epsilon,\quad \text{ for all }u\in\Omega_\LL, v\in\Omega_\RR, \text{ and }i\in\intZ.
      \end{equation}
      Let $\Gamma$ and $\rC$ be contours in $\Omega_\LL$ and $\Omega_\RR$, respectively, such that $\Gamma$ encloses all the points $\alpha_i$ and $\rC$ encloses all the points $\beta_i$.  We have the following Fredholm determinant formula for $\ELPP_{\boldsymbol{\lambda}}(m,n)$
    \begin{equation}
        \prob\left(\ELPP_{\boldsymbol{\lambda}}(m,n)\le L\right) = \det(\rI - K),
    \end{equation}
    where $K$ is an integral kernel acting on $L^2(\Gamma,\rd u/2\pi\ri)$ given by 
    \begin{equation}
    \label{eq:kernel_path-to-point}
        K(u_1,u_2) = \oint_\rC  \frac{\rd v}{2\pi\ri}\prod_{\ell=1}^n \frac{u_2-\beta_\ell}{v-\beta_\ell} \cdot \prod_{k=\lbd}^m \frac{v-\alpha_k}{u_2-\alpha_k}\cdot \frac{e^{L(u_2-v)}}{u_2-v}\cdot \ch(v,u_1),
    \end{equation}
    $\lbd$ is any integer parameter, and $\ch$ is the inhomogeneous characteristic function of $\boldsymbol{\lambda}$ defined in Theorem \ref{thm:hitting} (with the same parameter $\lbd$).
\end{thm}
\begin{rmk}
    When $\lambda_1=\cdots=\lambda_n$, by taking $\lbd=\lambda_1+1$ and using Remark \ref{rmk:step_IC}, we obtain 
    \begin{equation}
        K(u_1,u_2) = \oint_\rC  \frac{\rd v}{2\pi\ri}\prod_{\ell=1}^n \frac{u_2-\beta_\ell}{v-\beta_\ell} \cdot \prod_{k=\lambda_1+1}^m \frac{v-\alpha_k}{u_2-\alpha_k}\cdot \frac{e^{L(u_2-v)}}{(u_2-v)(v-u_1)}.
    \end{equation}
\end{rmk}
\begin{proof}[Proof of Theorem \ref{thm:LPP-path-to-point}]
    We take the integral of \eqref{eq:LPP_path-to-line} iteratively over $x_n,\ldots,x_1$ within the domain $\{0\le x_1\le \cdots\le x_n\le L\}$. Note that the bounds of $x_i$ are $x_{i-1}$ and $L$, for $i=n,\ldots,1$, with the convention $x_0=0$. The $i$-th row after the integration becomes
    \begin{equation*}
        \oint_{|z|=R} \frac{\rd z}{2\pi\ri }  \left(e^{L(z-\beta_i)} -e^{x_{i-1}(z-\beta_i)} \right)
    \frac{\prod_{k=1}^{i-1} (z-\beta_k)}{\prod_{k=1}^j (z-\beta_k)} \prod_{k=\lambda_j+1  }^m \frac{\beta_j-\alpha_k}{z-\alpha_k}.
    \end{equation*}
    The integral of the part including the $x_{i-1}$ term can be dropped without affecting the resulting determinant since it is the same as the $-e^{x_{i-1}(\beta_{i-1}-\beta_i)}M_{i-1,j}$ when $i\ge 2$, or vanishes after deforming the contour to infinity when $i=1$. As a result, we obtain 
    \begin{equation}
    \begin{split}        
        \prob\left(\ELPP_{\boldsymbol{\lambda}}(m,n)\le L\right) &= \det\left[ \oint_{|z|=R} \frac{\rd z}{2\pi\ri}  e^{L(z-\beta_i)}
    \frac{\prod_{k=1}^{i-1} (z-\beta_k)}{\prod_{k=1}^j (z-\beta_k)} \prod_{k=\lambda_j+1  }^m \frac{\beta_j-\alpha_k}{z-\alpha_k}\right]_{i,j=1}^n\\
    &= \det\left[ \oint_{|z|=R} \frac{\rd z}{2\pi\ri}  e^{L(z-\beta_i)}
    \frac{\prod_{k=1}^{i-1} (z-\beta_k)}{\prod_{k=1}^j (z-\beta_k)} \prod_{k=\lbd }^m \frac{\beta_i-\alpha_k}{z-\alpha_k}\cdot \prod_{k=\lbd}^{\lambda_j}\frac{z-\alpha_k}{\beta_j-\alpha_k}\right]_{i,j=1}^n,
        \end{split}
    \end{equation}
    where the second equality follows from the fact that $\prod_{j=1}^n \prod_{k=\lambda_j+1}^m(\beta_j-\alpha_k) = \prod_{i=1}^n\prod_{k=\lbd}^m(\beta_i-\alpha_k)\cdot \prod_{j=1}^n\prod_{k=\lbd}^{\lambda_j}(\beta_j-\alpha_k)^{-1}$.
    Now we change the $|z|=R$ contour to two disjoint contours $\Gamma\subset \Omega_\LL$ and $\rC\subset\Omega_\RR$ such that $\Gamma$ encloses all the points in $\{\alpha_k:\lambda_n+1\le k\le m\}$ and  $\rC$ encloses the points in $\{\beta_k:1\le k\le n\}$. We apply Theorem \ref{prop:main} by choosing
\begin{equation}
\begin{split}
    &p_i(v)=e^{L(v-\beta_i)}\cdot \prod_{k=\lbd}^m \frac{\beta_i-\alpha_k}{v-\alpha_k},\quad f_j(v)=\prod_{k=\lbd}^{\lambda_j}\frac{v-\alpha_k}{\beta_j-\alpha_k},\\
    &q_i(u)=e^{L(u-\beta_i)}\cdot \prod_{k=\lbd}^m \frac{\beta_i-\alpha_k}{u-\alpha_k}\cdot \prod_{k=1}^{i-1}(u-\beta_k),\quad g_j(u)=\prod_{k=\lbd}^{\lambda_j}\frac{u-\alpha_k}{\beta_j-\alpha_k}\cdot \frac{1}{\prod_{k=1}^j(u-\beta_k)}.
    \end{split}
\end{equation}
Note that for this case, the kernel $K$ can be simplified as discussed in Section \ref{sec:check}, more explicitly, $K$ is given by \eqref{eq:simplified_kernel2} with $Q(v,u)=\prod_{k=\lbd}^m\frac{v-\alpha_k}{u-\alpha_k}\cdot e^{L(u-v)}$. Note that the function $H(v,u)=-\ch(v,u)$ satisfies the definition of the inhomogeneous characteristic function as in Definition \ref{def:inhomogeneous_ch}, and the negative sign is moved outside the kernel $K$. This implies the formula \eqref{eq:kernel_path-to-point}.
\end{proof}

The structure of the paper is as follows. In Section \ref{sec:trace_class}, we verify that $K$ is a trace class operator; therefore, $\det(\rI+K)$ in Theorem \ref{prop:main} is well defined. Section \ref{sec:proof} is devoted to the proof of Theorem \ref{prop:main}. In Section \ref{sec:proof_ch}, we verify the well-definedness of the transition probability \eqref{eq:transition_prob} and the expression \eqref{eq:ch_hitting_exp}, then we prove Theorem \ref{thm:hitting}. Finally, Section \ref{sec:applications} presents two additional applications of Theorem \ref{prop:main}.

\section*{Acknowledgement}
The authors thank Jinho Baik, Percy Deift, Yuchen Liao, and Leo Petrov for the helpful discussions and suggestions. Z.L. was partly supported by  NSF grants DMS-2552334 and DMS-2246683. T.T. was partly supported by NSF grant DMS-2246683. Most of the work was completed while Z.L. was at the University of Kansas.

%%%%%%%%%%%%%%%%%%%%%%%%%%%%%%%%%%%%%%%%%%%%%%%%%%%%%%%%%%%%%%%%
\section{Verification of a trace class operator}
\label{sec:trace_class}

Note that $K$ is a finite rank operator. It suffices to show that the following function is square integrable
\begin{equation}
    L_\ell(u):= \oint_{\rC_1} \ddbar{v_1}{}\cdots \oint_{\rC_\ell}\ddbar{v_\ell}{} \frac{H(v_1,u)}{\left[\prod_{i=1}^{\ell-1}(v_i-v_{i+1})\right]\cdot \left[\prod_{i=1}^\ell (v_i-\beta_i)P_i(v_i)\right]},
\end{equation}
where the integration contours are the same as in Theorem \ref{prop:main}. We shrink the $v_1$ contour so that it lies inside all the $v_2,\ldots,v_\ell$ contours. This results in two terms: one term with the same expression but with a modified $v_1$ contour, and a second term consisting of $\ell-1$ integrals with the $v_1$ integral evaluated at the pole $v_1=v_2$. For the second term, we then shrink the $v_2$ contour and repeat this procedure. This yields
\begin{equation}
    L_\ell(u) = L_\ell^{(1)}(u) +L_\ell^{(2)}(u)+\cdots + L_\ell^{(\ell)}(u).
\end{equation}
The functions $L_\ell^{(i)}$ are given by
\begin{equation}
\begin{split}
    &L_\ell^{(i)}(u) \\
    &= \oint_{\rC_{\ell+1}} \ddbar{v_i}{} \prod_{j=i+1}^\ell \oint_{\rC_{j}} \ddbar{v_{j}}{}  \frac{H(v_i,u)}{\left[\prod_{j=i}^{\ell-1} (v_j-v_{j+1})\right]\cdot \left[\prod_{j=i+1}^\ell (v_j-\beta_j)P_j(v_j)\right] \cdot \prod_{j=1}^i(v_i-\beta_j)p_i(v_i)},
\end{split}
\end{equation}
where $\rC_{\ell+1}$ is a simple closed contour containing all $\beta_k$ and is itself contained in $\rC_\ell$. We claim that each $L_\ell^{(i)}$ is square integrable. In fact, when $i=\ell$, we have 
\begin{equation}
\label{eq:L_ell^ell}
    L_\ell^{(\ell)}(u) =\oint_{\rC} \frac{H(v,u)}{\prod_{j=1}^\ell(v-\beta_j) p_\ell(v)}\ddbar{v}{} =\sum_{k=1}^\ell c_{\ell,k} \oint_{\rC} \frac{f_k(v)}{\prod_{j=1}^k(v-\beta_j)} H(v,u) \ddbar{v}{} =\sum_{k=1}^\ell c_{\ell,k} g_k(u),
\end{equation}
where $c_{\ell,k}$ are the coefficients in the series expansion 
\begin{equation}
\label{eq:expansion_for_squareintegrability}
    \frac{1}{\prod_{j=1}^\ell(v-\beta_j) p_\ell(v)} =\sum_{k=1}^{\ell} c_{\ell,k} \frac{f_k(v)}{\prod_{j=1}^k(v-\beta_j)} + \text{a function of $v$ analytic in $\Omega$}.
\end{equation}
Note that such an expansion exists since $p_\ell$ and $f_k$ are analytic in $\Omega$, $p_\ell$ is nonzero in $\Omega$, and $f_k(\beta_k)=1$ for each $k$.

Recall that $g_j$ is square integrable on $\Gamma$. Thus \eqref{eq:L_ell^ell} implies that $L^{(\ell)}_\ell(u)$ is square integrable.

The case when $i<\ell$ can be proved similarly, since $\tilde p_i(v):=(v-v_{i+1})p_i(v)$ is a nonzero analytic function within the contour $\rC_{\ell+1}$.  By using an expansion similar to  \eqref{eq:expansion_for_squareintegrability} for $\tilde p_i(v)$, we know that 
\begin{equation}
\label{eq:expansion_for_squareintegrability2}
    \frac{1}{\prod_{j=1}^i(v-\beta_j) \tilde p_i(v)} =\sum_{k=1}^{i} \tilde c_{i,k} \frac{f_k(v)}{\prod_{j=1}^k(v-\beta_j)} + \text{a function of $v$ analytic within $\rC_{\ell+1}$},
\end{equation}
and 
\begin{equation}
    \oint_{\rC_{\ell+1}}\frac{\rd v_i}{2\pi\ri} \frac{H(v_i,u)}{(v_i-v_{i+1})\prod_{j=1}^i(v_i-\beta_j)p_i(v_i)}=\sum_{k=1}^i \tilde c_{i,k}g_k(u)
\end{equation}
for some coefficients $\tilde c_{i,k}= \tilde c_{i,k}(v_{i+1})$ depending on $v_{i+1}$. Note that these coefficients are determined by \eqref{eq:expansion_for_squareintegrability2} and they are uniformly bounded when $v_{i+1}\in\rC_{i}$. Hence the above function is square integrable, and its $L^2$ norm is bounded by a constant independent of $v_{i+1}$. This implies that $L_\ell^{(i)}(u)$ is square integrable. This completes the proof.

\section{Proof of Theorem \ref{prop:main}}
\label{sec:proof}

The starting point of the proof of Theorem \ref{prop:main} is the following Proposition \ref{lem:inverseofmatrix}, which provides a formula for the inverse of an upper-triangular matrix whose entries are defined via contour integrals.

\begin{prop} 
\label{lem:inverseofmatrix}
Let $n \ge 1$. Let $\Omega$ be a simply connected subset of the complex plane containing the points $\beta_1, \ldots, \beta_n$, and let $\rC$ be a simple closed contour in $\Omega$ that encloses the points $\beta_1, \ldots, \beta_n$. Let $p_i$ be analytic and nonzero in $\Omega$ for all $i$. Define an $n \times n$ matrix $U_0$ by
\begin{equation}
\label{eq:Sdef}
U_0(i,j) := \oint_{\rC} \ddbar{v}{} \, \frac{\prod_{\ell=1}^{i-1} (v-\beta_\ell)}{\prod_{\ell=1}^{j} (v-\beta_\ell)} p_i(v), \quad 1 \le i,j \le n.
\end{equation}
Then we have
\begin{equation}
    U_0^{-1} = W,
\end{equation}
where $W$ is an $n \times n$ matrix whose entries are given by
\begin{equation}
\label{eq:Wdef}
W(i,j) =   \oint_{\rC_1} \ddbar{v_{1}}{} \cdots  \oint_{\rC_j} \ddbar{v_{j}}{} \frac{\prod_{\ell=1}^{i-1} (v_1-\beta_\ell)}{ \left[\prod_{\ell=1}^{j-1}(v_{\ell} - v_{\ell+1})\right] \cdot \left[\prod_{\ell=1}^j (v_\ell-\beta_\ell) P_\ell(v_\ell)\right]}, \quad 1 \le i,j \le n.
\end{equation}
Here $P_1(v):=p_1(v)$, and $P_i(v):=p_i(v)/p_{i-1}(v)$ for $2\le i\le n$, and $\rC_1,\ldots,\rC_n$ are nested simple closed contours in $\Omega$, ordered from outermost to innermost. These notations are the same as in Theorem \ref{prop:main}.
\end{prop}

The proof will use the following simple lemma of changing the order of contours. Although not explicitly stated, this lemma also appeared in \cite{Liu-Ortiz25}, and will be applied frequently in this paper.
\begin{lm}
    \label{lem:change_order_contour}
    Assume $\annu$ is an annulus. Suppose that $\Sigma^-_1,\ldots,\Sigma^-_n$ and  $\Sigma^+_1,\ldots,\Sigma^+_n$ are positively oriented circles in $\annu$ that are concentric with the annulus. Assume the circles $\Sigma_i^-$ are ordered from outermost to innermost, while the circles $\Sigma_i^+$ are ordered from innermost  to outermost. The superscripts indicate whether the sequence of radii is decreasing ($-$) or increasing ($+$). Let $h_1,\ldots,h_n$ be functions analytic in $\annu$. Then the following identity holds:
    \begin{equation}
        \label{eq:change_order_contour}
        \oint_{\Sigma^-_1}\frac{\rd w_1}{2\pi\ri} \cdots \oint_{\Sigma^-_n} \frac{\rd w_n}{2\pi\ri} \frac{\prod_{\ell=1}^n h_\ell(w_\ell)}{ \prod_{\ell=1}^{n-1} (w_\ell- w_{\ell+1})}
        =\sum_{k=1}^n \sum_{1\le i_1<\cdots<i_k = n} \oint_{\Sigma^+_{i_1}}\frac{\rd w_{i_1}}{2\pi\ri} \cdots \oint_{\Sigma^+_{i_k}}\frac{\rd w_{i_k}}{2\pi\ri}
        \frac{\prod_{\ell=1}^k h_{i_{\ell-1}\to i_\ell}(w_{i_\ell})}{\prod_{\ell=1}^{k-1}(w_{i_\ell}-w_{i_{\ell+1}})},
    \end{equation}
    where $h_{a\to b}(w):= \prod_{i=a+1}^b h_i(w)$ for any $a,b\in\intZ$ satisfying $a\le b$, and $i_0$ is set to be $0$ in the factor $h_{i_0\to i_1}$ on the right-hand side of the equation above.
\end{lm}
\begin{proof}[Proof of Lemma \ref{lem:change_order_contour}]
    Without loss of generality we assume that  the contours $\Sigma_i^{-}$ are all inside the contours $\Sigma_i^{+}$ (otherwise we could shrink the  $\Sigma_i^{-}$ contours sufficiently close to the inner circle of $\annu$ without encountering any poles). We deform the contour $\Sigma_\ell^-$ to $\Sigma_\ell^+$ for $\ell=1,\ldots,n$ successively. Note that the $w_\ell$-integral produces two contributions: one from the residue of the integrand at $w_\ell = w_{\ell-1}$, and the other from the integral along the deformed contour $\Sigma_\ell^+$. Collecting and organizing these terms, we obtain \eqref{eq:change_order_contour}.
\end{proof}

\begin{proof}[Proof of Proposition \ref{lem:inverseofmatrix}]

Let $\rC^\out_n, \ldots, \rC^\out_1, \rC^\inn_1, \ldots, \rC^\inn_n$ be $2n$ nested simple closed contours in $\Omega$, each enclosing the points $\beta_1, \ldots, \beta_n$, ordered from outermost to innermost. Here we also assume that $\rC_i^\inn=\rC_i$ for $1\le i\le n$.

Applying Lemma \ref{lem:change_order_contour}, we can rewrite the entry $W(i,j)$ as follows
\begin{equation}
\label{eq:W0res}
W(i, j) = \sum_{k=1}^j \sum_{1 \le i_1 < \cdots < i_k = j} \prod_{\ell=1}^k \int_{\rC^\out_{i_\ell}} \ddbar{v_{i_\ell}}{}  \frac{\prod_{\ell=1}^{i-1} (v_{i_1}-\beta_\ell)}{ \left[\prod_{\ell=1}^{k-1}(v_{i_\ell} - v_{i_{\ell+1}})\right] \cdot \left[\prod_{\ell=1}^k \qq_{i_{\ell-1}\to i_\ell}(v_{i_\ell})\right]}
\end{equation}
with $i_0=0$ and $\qq_{a}(v):= (v-\beta_a)P_a(v)$. Recall the notation $h_{a\to b}$ we used in Lemma \ref{lem:change_order_contour}, thus
\begin{equation}
    \qq_{i_{\ell-1}\to i_\ell}(v_{i_\ell})= \prod_{a=i_{\ell-1} + 1}^{i_\ell} (v_{i_\ell}-\beta_a)  P_a(v_{i_\ell}).
\end{equation}
We also choose the specific contour $\rC_1^\inn$ for $v$ in $U_0(i,j)$ and have
\begin{equation}
    U_0(i,j)=\oint_{\rC_1^\inn} \ddbar{v}{}\frac{\prod_{\ell=1}^{i-1} (v-\beta_\ell)}{\prod_{\ell=1}^{j} (v-\beta_\ell)}p_i(v).
\end{equation}

We use the above expressions to evaluate
\begin{equation*}
(U_0 W)(i,j) = \sum_{m=1}^n U_0(i,m)\, W(m,j).
\end{equation*}
Note that the index $m$ only appears in $\frac{1}{\prod_{\ell=1}^{m} (v-\beta_\ell)}$ and $\prod_{\ell=1}^{m-1} (v_{i_1}-\beta_\ell)$. Applying Lemma \ref{lm:telescope}, we obtain 
\begin{equation}
\begin{split}
\sum_{m=1}^n \frac{\prod_{\ell=1}^{m-1} (v_{i_1}-\beta_\ell)}{\prod_{\ell=1}^{m} (v-\beta_\ell)}
&= \frac{1}{v-v_{i_1}}  \left[1-\frac{\prod_{\ell=1}^{n} (v_{i_1}-\beta_\ell)}{\prod_{\ell=1}^{n} (v-\beta_\ell)}\right].
\end{split}
\end{equation}
Inserting this summation, we obtain
\begin{equation*}
\begin{split}
&(U_0W)(i,j)\\
&= \sum_{k=1}^j \sum_{1  \le i_1 < \cdots < i_k = j} \oint_{\rC^\inn_{1}} \ddbar{v}{}    \prod_{\ell=1}^k \int_{\rC^\out_{i_\ell}} \ddbar{v_{i_\ell}}{}  \frac{\prod_{\ell=1}^{i-1} (v-\beta_\ell) \cdot p_i(v)}{ \left[\prod_{\ell=1}^{k-1}(v_{i_\ell} - v_{i_{\ell+1}}) \right] \cdot \left[\prod_{\ell=1}^k \qq_{i_{\ell-1}\to i_\ell}(v_{i_\ell}) \right]}   \frac{1}{v-v_{i_1}} \\
&-\sum_{k=1}^j \sum_{1 \le i_1 < \cdots < i_k = j} \oint_{\rC^\inn_{1}} \ddbar{v}{}    \prod_{\ell=1}^k \int_{\rC^\out_{i_\ell}} \ddbar{v_{i_\ell}}{}  \frac{\prod_{\ell=1}^{i-1} (v-\beta_\ell) \cdot p_i(v)}{ \left[\prod_{\ell=1}^{k-1}(v_{i_\ell} - v_{i_{\ell+1}}) \right] \cdot \left[\prod_{\ell=1}^k \qq_{i_{\ell-1}\to i_\ell}(v_{i_\ell})\right]}   \frac{1}{v-v_{i_1}}\frac{\prod_{\ell=1}^{n} (v_{i_1}-\beta_\ell)}{\prod_{\ell=1}^{n} (v-\beta_\ell)}. 
\end{split}
\end{equation*}
The first term vanishes since the integrand is analytic in $v$. Thus, $(U_0 W)(i,j)$ is equal to
\begin{equation}
\label{eq:analyticv}
\sum_{k=1}^j \sum_{1  \le i_1 < \cdots < i_k = j} \oint_{\rC^\inn_{1}} \ddbar{v}{}  \prod_{\ell=1}^k \int_{\rC^\out_{i_\ell}} \ddbar{v_{i_\ell}}{}  \frac{\prod_{\ell=1}^{i-1} (v-\beta_\ell) \cdot p_i(v)}{ \left[\prod_{\ell=1}^{k-1}(v_{i_\ell} - v_{i_{\ell+1}}) \right] \cdot \left[\prod_{\ell=1}^k \qq_{i_{\ell-1}\to i_\ell}(v_{i_\ell}) \right]}   \frac{1}{v_{i_1}-v} \frac{\prod_{\ell=1}^{n} (v_{i_1}-\beta_\ell)}{\prod_{\ell=1}^{n} (v-\beta_\ell)}.
\end{equation}
Now we evaluate the $v_{i_1}$-integral. Due to the $\prod_{\ell=1}^{n} (v_{i_1}-\beta_\ell)$ factor, it is analytic at $v_{i_1}=\beta_\ell$ and hence only the residue at $v_{i_1}=v$ survives (note that $v_{i_1}-v_{i_2}$ in the denominator does not contribute to any pole in this evaluation since $v_{i_2}$ lies outside $\rC_{i_1}^\out$). We rewrite $v$ as $v_{i_1}$ and deform the contour back to $\rC_{i_1}^\out$. This gives
\begin{equation}
\label{eq:SimpleS0W0}
\begin{split}
(U_0 W)(i,j)&= \sum_{k=1}^j \sum_{1  \le i_1 < \cdots < i_k = j} \prod_{\ell=1}^k \int_{\rC^\out_{i_\ell}} \ddbar{v_{i_\ell}}{} \frac{\prod_{\ell=1}^{i-1} (v_{i_1}-\beta_\ell) \cdot p_{i}(v_{i_1})}{\left[\prod_{\ell=1}^{k-1}(v_{i_\ell} - v_{i_{\ell+1}})\right]\cdot \left[\prod_{\ell=1}^k \qq_{i_{\ell-1}\to i_\ell}(v_{i_\ell})\right]} \\
&= \prod_{\ell=1}^j \int_{\rC^\inn_{\ell}} \ddbar{v_{\ell}}{} \frac{\prod_{\ell=1}^{i-1} (v_{1}-\beta_\ell) \cdot p_i(v_1)}{ \left[\prod_{\ell=1}^{j-1}(v_{\ell} - v_{\ell+1})\right] \cdot \left[\prod_{\ell=1}^j (v_\ell-\beta_\ell) P_\ell(v_\ell)\right]},
\end{split}
\end{equation}
where, in the second equality, we applied Lemma \ref{lem:change_order_contour}.

It remains to show the right hand side of \eqref{eq:SimpleS0W0} is equal to $\delta_i(j)$. We take residues successively at $v_1 = v_2$, $v_2 = v_3$, and so on. If $i \ge j$, we take all residues starting from $v_1 = v_2$ and ending with $v_{j-1} = v_j$. Using $\prod_{\ell=1}^j P_\ell(v) = p_j(v)$, we find that
\begin{equation*}
   (U_0 W)(i,j) = \int_{\rC^\inn_{j}} \ddbar{v_j}{} \frac{\prod_{\ell=1}^{i-1} (v_{j}-\beta_\ell) \cdot p_i(v_j)}{\prod_{\ell=1}^{j} (v_{j}-\beta_\ell) \cdot p_j(v_j)} = \delta_i(j), \quad \text{for }i \ge j.
\end{equation*}
If $i < j$, we take residues starting from $v_1 = v_2$ and ending with $v_{i-1} = v_i$. Noting that $\prod_{\ell=1}^{i} P_\ell(v) = p_i(v)$, we obtain
\begin{equation*}
\begin{split}
(U_0 W)(i,j) &= \left(\prod_{\ell=i}^j \int_{\rC^\inn_{\ell}} \ddbar{v_{\ell}}{} \right)
\frac{\prod_{\ell=1}^{i-1} (v_{i}-\beta_\ell) \cdot p_i(v_i)}{\left[\prod_{\ell=i}^{j-1}(v_{\ell} - v_{\ell+1})\right] \cdot \left[\prod_{\ell=1}^{i} (v_{i}-\beta_\ell) \prod_{\ell=1}^{i} P_\ell(v_i)\prod_{\ell=i+1}^j (v_\ell-\beta_\ell) P_\ell(v_\ell)\right]} \\
&= \left(\prod_{\ell=i}^j \int_{\rC^\inn_{\ell}} \ddbar{v_{\ell}}{} \right)
\frac{1}{\left[\prod_{\ell=i}^{j-1}(v_{\ell} - v_{\ell+1})\right] \cdot \left[(v_i-\beta_i)\prod_{\ell=i+1}^j (v_\ell-\beta_\ell) P_\ell(v_\ell)\right]}.
\end{split}
\end{equation*}
The resulting integrand is of order $O(v_i^{-2})$ as a function of $v_i$. Therefore, we can deform the outermost contour $\rC^\inn_{i}$ to infinity, which implies that $(U_0 W)(i,j) = 0$ for all $i < j$. Hence, the proposition follows.
\end{proof}

\medskip

The remaining part of this section is devoted to proving Theorem \ref{prop:main}.

\medskip 

\textit{Proof of Theorem \ref{prop:main}.} Write the matrix inside the determinant on the left-hand side of \eqref{eq:findettoFred} as $U(i,j) + M(i,j)$, where
\begin{equation}
U(i,j) := \oint_{\rC} \frac{\prod_{\ell=1}^{i-1} (v-\beta_\ell)}{\prod_{\ell=1}^{j} (v-\beta_\ell)} p_i(v) f_j(v)\,\ddbar{v}{}, 
\quad 
M(i,j) := \int_{\Gamma} q_i(u) g_j(u)\,\ddbar{u}{}.
\end{equation}
Note that $U$ is upper triangular with $1$'s on the diagonal since $p_i$ and $f_j$ are analytic within $\rC$ and  $p_i(\beta_i)=f_i(\beta_i)=1$. Now we write
\begin{equation}
\label{eq:AinCB^-2}
\det(U+M) = \det(\mathrm{I} + M U^{-1}).
\end{equation}
We view $M$ as the product of two operators $K_1$ and $K_2$:
\begin{equation}
\label{eq:def_K1K2}
    M(i,j) = \int_{\Gamma} K_1(i,u)K_2(u,j)\,\ddbar{u}{},\quad \text{with}\quad  K_1(i,u):= q_i(u), \quad K_2(u,j):= g_j(u).
\end{equation}
Therefore we further write \eqref{eq:AinCB^-2} as
\begin{equation}
    \label{eq:AinCB^-1}
    \det(U+M) = \det(\mathrm{I} + K_2 U^{-1} K_1),
\end{equation}
where $K_2U^{-1}K_1$ is an integral operator on $L^2(\Gamma,|\rd u|)$ which is a trace class operator since both $q_i$ and $g_i$ are square integrable functions.

Define an $n \times n$ matrix $U_1$ by
\begin{equation}
\label{eq:Rdef}
U_1(i,j) := \oint_{\rC} \ddbar{v}{} \, \frac{\prod_{\ell=1}^{i-1} (v-\beta_\ell)}{\prod_{\ell=1}^{j} (v-\beta_\ell)} f_j(v), 
\quad 1 \le i,j \le n.
\end{equation}
Note that $U_1$ is also upper triangular with $1$'s on the diagonal.  Now we define a family of polynomials $\{\bbeta_i(v)\}_{i=1}^n$ with degree less than $n$ as follows:
\begin{equation}
\label{eq:alphadef}
\bbeta_i(v) = \sum_{k=1}^n U_1^{-1}(i,k) \prod_{\ell=1}^{k-1} (v-\beta_\ell), \quad 1 \le i \le n,
\end{equation}
where $U_1^{-1} = (U_1^{-1}(i,j))_{i,j=1}^n$. By definition, it follows that 
\begin{equation}
\label{eq:orthogonal}
\oint_{\rC} \bbeta_i(v) \frac{f_j(v)}{\prod_{\ell=1}^{j} (v-\beta_\ell)} \ddbar{v}{} = \sum_{k=1}^m U_1^{-1}(i,k)U_1(k,j)= \delta_i(j), \quad \text{for } 1 \le i,j \le n.
\end{equation}

\medskip 

Recall the $n \times n$ matrix $U_0$ defined in \eqref{eq:Sdef}. For $1 \le i,j \le n$, using the summation formula in Lemma \ref{lm:telescope}, we have
\begin{equation*}
\begin{split}
(U_0U_1)(i,j) &= \sum_{m=1}^n \oint_{\rC_2} \ddbar{v_2}{}  \frac{\prod_{\ell=1}^{i-1} (v_2-\beta_\ell)}{\prod_{\ell=1}^{m} (v_2-\beta_\ell)} p_i(v_2) \oint_{\rC_1} \ddbar{v_1}{}  \frac{\prod_{\ell=1}^{m-1} (v_1-\beta_\ell)}{\prod_{\ell=1}^{j} (v_1-\beta_\ell)} f_j(v_1) \\
&= \oint_{\rC_2} \ddbar{v_2}{}  \oint_{\rC_1} \ddbar{v_1}{}  \frac{\prod_{\ell=1}^{i-1} (v_2-\beta_\ell) \cdot p_i(v_2)f_j(v_1)}  {\prod_{\ell=1}^{j} (v_1-\beta_\ell)}  \cdot \frac{1}{v_2 - v_1}\left(1 - \frac{\prod_{\ell=1}^{n} (v_1-\beta_\ell)}{\prod_{\ell=1}^{n} (v_2-\beta_\ell)}\right)\\
&=\oint_{\rC_2} \ddbar{v_2}{}  \oint_{\rC_1} \ddbar{v_1}{}  \frac{\prod_{\ell=1}^{i-1} (v_2-\beta_\ell) \cdot p_i(v_2)f_j(v_1)}  {\prod_{\ell=1}^{j} (v_1-\beta_\ell)}  \cdot \frac{1}{v_2 - v_1}\\
&\quad + \oint_{\rC_2} \ddbar{v_2}{}  \oint_{\rC_1} \ddbar{v_1}{}  \frac{\prod_{\ell=1}^{i-1} (v_2-\beta_\ell) \cdot p_i(v_2)f_j(v_1)}  {\prod_{\ell=1}^{j} (v_1-\beta_\ell)}  \cdot \frac{1}{v_1 - v_2}  \frac{\prod_{\ell=1}^{n} (v_1-\beta_\ell)}{\prod_{\ell=1}^{n} (v_2-\beta_\ell)},
\end{split}
\end{equation*}
where $\rC_1$ encloses $\rC_2$, and both contours enclose the points $\beta_1,\ldots,\beta_n$. 
The first term vanishes since the integrand is analytic as a function of $v_2$ within $\rC_2$. For the second term, the integrand as a function of $v_1$ only has one pole at $v_1=v_2$. By evaluating the pole, we obtain
\begin{equation*}
\begin{split}
(U_0U_1)(i,j) 
= \oint_{\rC_2} \ddbar{v_2}{}  \frac{\prod_{\ell=1}^{i-1} (v_2-\beta_\ell)}{\prod_{\ell=1}^{j} (v_2-\beta_\ell)} p_i(v_2) f_j(v_2) 
= U(i,j).
\end{split}
\end{equation*}
Now Proposition \ref{lem:inverseofmatrix} implies that $U^{-1} = U_1^{-1}W$, where $W$ is an $n \times n$ matrix defined in \eqref{eq:Wdef}. Hence, for $1 \le i,j \le n$, we have
\begin{equation}
\label{eq:B^-1def}
U^{-1}(i,j) = \oint_{\rC_1} \ddbar{v_{1}}{} \cdots  \oint_{\rC_j} \ddbar{v_{j}}{} \frac{\bbeta_i(v_1)}{ \left[\prod_{\ell=1}^{j-1}(v_{\ell} - v_{\ell+1})\right] \cdot \left[\prod_{\ell=1}^j (v_\ell-\beta_\ell) P_\ell(v_\ell)\right]},
\end{equation}
where we used the definition of $\bbeta_i$ in \eqref{eq:alphadef}. Now we combine it with the definition of the kernel $K_2$ in \eqref{eq:def_K1K2}, and obtain
\begin{equation}
\label{eq:CB^-1expr}
\begin{split}
(K_2U^{-1})(u,j) = \oint_{\rC_1} \ddbar{v_{1}}{} \cdots  \oint_{\rC_j} \ddbar{v_{j}}{} \frac{1}{ \left[\prod_{\ell=1}^{j-1}(v_{\ell} - v_{\ell+1})\right] \cdot \left[\prod_{\ell=1}^j (v_\ell-\beta_\ell) P_\ell(v_\ell)\right]} \sum_{m=1}^n g_m(u) \bbeta_m(v_1).
\end{split}
\end{equation}
Recall in the assumption of Theorem \ref{prop:main}, $H(v,u)$ is analytic in $v \in \Omega$. We also recall \eqref{eq:alphadef}. $\bbeta_i(v)=\prod_{\ell=1}^{i-1} (v-\beta_\ell)+\sum_{k=i+1}^n U_1^{-1}(i,k) \prod_{\ell=1}^{k-1} (v-\beta_\ell)$ since $U_1^{-1}$ is upper triangular with $1$'s on the diagonal. $H(v,u)$ has the following series expansion in $\Omega$
\begin{equation}
\label{eq:H_expansion}
    H(v,u) = \sum_{m=1}^n c_m(u)\bbeta_m(v) +J(v,u), 
\end{equation}
where $J(v,u)$ is a function divisible by $\prod_{\ell=1}^{n} (v-\beta_\ell)$, i.e., $J(v,u) \cdot \prod_{\ell=1}^{n} (v-\beta_\ell)^{-1}$ is analytic in $v \in \Omega$. Recall that $H(v,u)$ satisfies \eqref{eq:fgHrel}. Inserting \eqref{eq:H_expansion} in \eqref{eq:fgHrel} implies
\begin{equation}
\begin{split}
    g_i(u)&= \oint_{\rC} H(v,u)\frac{f_i(v)}{\prod_{\ell=1}^{i} (v-\beta_\ell)}\ddbar{v}{} \\
    &=\sum_{m=1}^n c_m(u)\oint_\rC \bbeta_m(v)\frac{f_i(v)}{\prod_{\ell=1}^{i} (v-\beta_\ell)}\ddbar{v}{} + \oint_\rC \frac{f_i(v)J(v,u)}{\prod_{\ell=1}^{i} (v-\beta_\ell)}\\
    &=c_i(u),
\end{split}
\end{equation}
where we applied the orthogonal relation \eqref{eq:orthogonal} and the assumption of $J$ in the last equation.  Now we have
\begin{equation}
\label{eq:H_expansion2}
    \sum_{m=1}^n g_m(u)\bbeta_m(v) = \sum_{m=1}^n c_m(u)\bbeta_m(v) = H(v,u) - J(v,u).
\end{equation}
Inserting it to \eqref{eq:CB^-1expr} gives
\begin{equation}
\label{eq:BA_inverse_aux}
\begin{split}
(K_2U^{-1})(u,j) &= \oint_{\rC_1} \ddbar{v_{1}}{} \cdots  \oint_{\rC_j} \ddbar{v_{j}}{} \frac{H(v_1,u)}{ \left[\prod_{\ell=1}^{j-1}(v_{\ell} - v_{\ell+1})\right] \cdot \left[\prod_{\ell=1}^j (v_\ell-\beta_\ell) P_\ell(v_\ell)\right]}\\
&\quad -  \oint_{\rC_1} \ddbar{v_{1}}{} \cdots  \oint_{\rC_j} \ddbar{v_{j}}{} \frac{J(v_1,u)}{ \left[\prod_{\ell=1}^{j-1}(v_{\ell} - v_{\ell+1})\right] \cdot \left[\prod_{\ell=1}^j (v_\ell-\beta_\ell) P_\ell(v_\ell)\right]} .  
\end{split}
\end{equation}

Now we claim that the second term on the right hand side of \eqref{eq:BA_inverse_aux} vanishes. In fact, this follows from the analyticity of $J(v,u)\prod_{\ell=1}^{n} (v-\beta_\ell)^{-1}$ in $v\in \Omega$. If we evaluate the $v_1$ integral, the only contribution comes from the residue at $v_1=v_2$. Then we evaluate the $v_2$ integral, due to the factor $J(v_2,u)$, only the residue at $v_2=v_3$ survives. We repeat this procedure and obtain
\begin{equation}
    \oint_{\rC_1} \ddbar{v_{1}}{} \cdots  \oint_{\rC_j} \ddbar{v_{j}}{} \frac{ J(v_1,u)}{ \left[\prod_{\ell=1}^{j-1}(v_{\ell} - v_{\ell+1})\right] \cdot \left[\prod_{\ell=1}^j (v_\ell-\beta_\ell) P_\ell(v_\ell)\right]} = \oint_{\rC_j} \ddbar{v_j}{} \frac{J(v_j,u)}{\prod_{\ell=1}^{j} (v_j-\beta_\ell)p_j(v_j)}=0,
\end{equation}
since the integrand is analytic in $\Omega$.

Thus \eqref{eq:BA_inverse_aux}  becomes
\begin{equation*}
\begin{split}
(K_2U^{-1})(u,j) &= \oint_{\rC_1} \ddbar{v_{1}}{} \cdots  \oint_{\rC_j} \ddbar{v_{j}}{} \frac{ H(v_1,u)}{ \left[\prod_{\ell=1}^{j-1}(v_{\ell} - v_{\ell+1})\right] \cdot \left[\prod_{\ell=1}^j (v_\ell-\beta_\ell) P_\ell(v_\ell)\right]}.
\end{split}
\end{equation*}
Inserting it into \eqref{eq:AinCB^-1} gives the desired result.

\section{Probabilistic representation of the inhomogeneous characteristic function}
\label{sec:proof_ch}
In this section, we first show the left-hand side of \eqref{eq:sum_normalized} is absolutely convergent and prove the identity \eqref{eq:sum_normalized}. Then we will verify that \eqref{eq:ch_hitting_exp} is well defined and absolutely convergent, and prove Theorem \ref{thm:hitting}.

\subsection{Well-definedness and verification of \eqref{eq:sum_normalized}}
\label{sec:verification_transition}

By applying Lemma \ref{lm:telescope}, we know that
\begin{equation}
\begin{split}
    \sum_{x\in\intZ} \left|\prob(G_{ k}=x\mid G_{k-1}=y)\right|
    &=\sum_{x\le y} \left| \prod_{i=x+1}^y \frac{\cc-\alpha_i}{\beta_k-\alpha_i} - \prod_{i=x}^y \frac{\cc-\alpha_i}{\beta_k-\alpha_i}\right|\\
    &\le 2\sum_{x\le y} \prod_{i=x+1}^y \left| \frac{\cc-\alpha_i}{\beta_k-\alpha_i}\right|\le 2\sum_{x\le y}(1-\epsilon)^{y-x}<\infty,
\end{split}
\end{equation}
where we used \eqref{eq:cc_assumption} in the second  inequality. Thus the left-hand side of \eqref{eq:sum_normalized} is absolutely convergent. Applying Lemma \ref{lm:telescope} and using \eqref{eq:cc_assumption}, we obtain \eqref{eq:sum_normalized} immediately.

\subsection{Well-definedness of the hitting expectation representation for $\ch$}
\label{sec:well-definedness-hitting}
Note that for any fixed $x$, the expectation is well-defined and finite due to the indicator function $\1_{\tau<n}$. 

When $x\le \lambda_n$, since $\lambda_1\ge \cdots\ge \lambda_n$ and $G_0\ge G_1\ge \cdots\ge G_{n-1}$, we have $G_k \le G_0=x\le \lambda_n\le \lambda_{k+1}$ for each $0\le k\le n-1$. This implies $\1_{\tau<n}=0$ and the expectation vanishes.

On the other hand, when $x> \lambda_1$, we have $\tau=0$ and $G_\tau=x$. Therefore
\begin{equation}
\begin{split}
    &\sum_{x> \lambda_1}\prod_{i=\lbd}^{x-1}(u-\alpha_i)\cdot \E_{G_0=x}\left[ \frac{\prod_{j=1}^\tau (v-\beta_j)}{\prod_{i=\lbd}^{G_\tau}(v-\alpha_i)}\frac{1}{\prod_{i=1}^\tau(\cc-\beta_i) \cdot \prod_{j=G_\tau+1}^x (\cc-\alpha_j)}\1_{\tau<n}\right]\\
    &=\sum_{x=\lambda_1+1}^\infty \prod_{i=\lbd}^{x-1}(u-\alpha_i)\cdot \frac{1}{\prod_{i=\lbd}^x(v-\alpha_i)}\\
    &=\frac{1}{v-u} \prod_{i=\lbd}^{\lambda_1}\frac{u-\alpha_i}{v-\alpha_i},
\end{split}
\end{equation}
where the last equation comes from Lemma \ref{lm:telescope} and the assumption \eqref{eq:assumption_beta}. Note that this summation is absolutely convergent due to \eqref{eq:assumption_beta}.

Combining the discussions above, we immediately know that the right-hand side of \eqref{eq:ch_hitting_exp} is well defined and the summation is absolutely convergent.

\subsection{Proof of Theorem \ref{thm:hitting}}
\label{sec:proof-hitting}

Recall that $\alpha_i\in\Omega_\LL$ and $\beta_i\in\Omega_\RR$, and $\Omega_\LL\cap\Omega_\RR=\emptyset$. It is straightforward to see that $\ch(v,u)$ is analytic for $v\in\Omega_\RR$ since each summand is analytic and the summation is absolutely convergent. 

We also note that Theorem \ref{thm:hitting} is independent of the parameter $\lbd$ since $\lbd$ cancels out in the integral equations \eqref{eq:char_linear_equations} if we insert the formula \eqref{eq:ch_hitting_exp}. Without loss of generality, we assume $\lbd<\lambda_n$ for simplification.

It remains to verify that $\ch$ satisfies the integral equations \eqref{eq:char_linear_equations}. Since the summation in $\ch(v,u)$ is absolutely convergent, we can change the order of the summation and integral. This implies
\begin{equation}
    \begin{split}
        \oint_\rC \ftn_m(v) \ch(v,u)\frac{\rd v}{2\pi\ri}
       = \sum_{x\in \intZ} \prod_{i=\lbd}^{x-1}(u-\alpha_i) \cdot \E_{G_0=x} \left[ M_\tau \cdot \1_{\tau<n} \right],
    \end{split}
\end{equation}
where 
\begin{equation}
    M_k:= \oint_\rC \ftn_m(v) \cdot Q_k(v) \frac{\rd v}{2\pi\ri},
    \quad \text{with}\quad Q_k(v):=\frac{\prod_{j=1}^k (v-\beta_j)}{\prod_{i=\lbd}^{G_k}(v-\alpha_i)} \frac{1}{\prod_{i=1}^k(\cc-\beta_i)\cdot \prod_{j=G_k+1}^{G_0}(\cc-\alpha_j)}.
\end{equation}

Denote $\hat M_k= M_k\cdot \1_{G_k> \lambda_m}$. Below we show that $(\hat M_k)_{0\le k\le m-1}$ is a martingale with respect to the filtration $\mathcal{F}_k:=\sigma(G_0,\ldots,G_k)$. In fact, for each $1\le k\le m-1$, we have 
\begin{equation}
\begin{split}
    &\E \left[\hat M_k \mid \mathcal{F}_{k-1}\right]\\
    &=\1_{G_{k-1}> \lambda_m}\sum_{y=\lambda_m+1}^{G_{k-1}}\oint_{\rC}\ftn_m(v) \cdot Q_{k-1}(v) \cdot \frac{(v-\beta_k)\prod_{i=y+1}^{G_{k-1}}(v-\alpha_i)}{(\cc-\beta_k)\cdot \prod_{j=y+1}^{G_{k-1}}(\cc-\alpha_j)} \cdot  (\beta_k-\cc)\frac{\prod_{i=y+1}^{G_{k-1}}(\cc-\alpha_i)}{\prod_{i=y}^{G_{k-1}}(\beta_k-\alpha_i)} \frac{\rd v}{2\pi\ri}\\
    &=\1_{G_{k-1}> \lambda_m}\oint_{\rC}\ftn_m(v) \cdot Q_{k-1}(v) \cdot (-v+\beta_k)\cdot \sum_{y=\lambda_m+1}^{G_{k-1}} \frac{\prod_{i=y+1}^{G_{k-1}}(v-\alpha_i)}{\prod_{i=y}^{G_{k-1}}(\beta_k-\alpha_i)} \frac{\rd v}{2\pi\ri}\\
    &=\hat M_{k-1}  -\1_{G_{k-1}>\lambda_m}\cdot  \oint_\rC \ftn_m(v)\cdot Q_{k-1}(v) \cdot \prod_{i=\lambda_m+1}^{G_{k-1}} \frac{v-\alpha_i}{\beta_k-\alpha_i}\frac{\rd v}{2\pi\ri},
\end{split}
\end{equation}
where we used Lemma \ref{lm:telescope} in the last equation. Note that the last integral is equal to
\begin{equation}
    \text{constant}\cdot \oint_{\rC} \frac{1}{\prod_{j=k}^m (v-\beta_j)}\frac{\rd v}{2\pi\ri}=0.
\end{equation}
Therefore $\E \left[\hat M_k \mid \mathcal{F}_{k-1}\right] = \hat M_{k-1}$ and $(\hat M_k )_{0\le k\le m-1}$ is a martingale.

Now we claim that 
\begin{equation}
\label{eq:condition_stopping_time}
    \E_{G_0=x}\left[\hat M_{\tau\wedge (m-1)}\right] = \E_{G_0=x} \left[ M_\tau\cdot \1_{\tau<n}\right].
\end{equation}
In fact, when $\tau\ge m$, by the definition of $M_\tau$ we see that the integrand is analytic for $v$ in $\Omega_\RR$, hence $M_\tau=0$. On the other hand, by the definition of $\tau$, $\tau\ge m$ implies $G_{m-1}\le \lambda_m$. Moreover, $\tau\wedge (m-1)=m-1$, which implies $\hat M_{\tau\wedge (m-1)}=\hat M_{m-1}$, which is equal to zero due to the indicator function $\1_{G_{m-1}>\lambda_m}=0$. These discussions imply \eqref{eq:condition_stopping_time} holds when $\tau\ge m$.

The other case when $\tau\le m-1$ is trivial since 
\begin{equation}
    \hat M_{\tau\wedge (m-1)} =\hat M_{\tau} =M_\tau\cdot \1_{G_\tau>\lambda_m} =M_\tau =M_\tau\cdot \1_{\tau<n}, 
\end{equation}
where in the second-to-last equation we used the fact that $G_\tau>\lambda_{\tau+1}\ge \lambda_m$.

Finally, we apply the optional stopping theorem and obtain
\begin{equation}
    \E_{G_0=x}\left[\hat M_{\tau\wedge (m-1)}\right] = \E_{G_0=x}\left[\hat M_0\right]
\end{equation}
which is evaluated explicitly as follows
\begin{equation}
    \begin{split}
       \E_{G_0=x}\left[\hat M_{\tau\wedge (m-1)}\right]&= \E_{G_0=x}\left[\hat M_0\right]
        =\1_{x>\lambda_m} \oint_\rC \ftn_m(v)\cdot Q_0(v) \frac{\rd v}{2\pi\ri}\\
        &=\1_{x>\lambda_m}\oint_{\rC}\ftn_m(v)\frac{1}{\prod_{i=\lbd}^x(v-\alpha_i)}\frac{\rd v}{2\pi\ri}.
    \end{split}
\end{equation}

Combining with \eqref{eq:condition_stopping_time}, we immediately obtain
\begin{equation}
    \begin{split}
        \oint_\rC \ftn_m(v) \ch(v,u)\frac{\rd v}{2\pi\ri}
       &= \sum_{x\in \intZ} \prod_{i=\lbd}^{x-1}(u-\alpha_i) \cdot\1_{x>\lambda_m}\oint_{\rC}\ftn_m(v)\frac{1}{\prod_{i=\lbd}^x(v-\alpha_i)}\frac{\rd v}{2\pi\ri}\\
       &= \oint_\rC \ftn_m(v)\sum_{x=\lambda_m+1}^\infty \frac{\prod_{i=\lbd}^{x-1}(u-\alpha_i)}{\prod_{i=\lbd}^x(v-\alpha_i)}\frac{\rd v}{2\pi\ri}\\
       &=\oint_\rC \ftn_m(v)\cdot \frac{1}{v-u} \cdot \prod_{i=\lbd}^{\lambda_m} \frac{u-\alpha_i}{v-\alpha_i}\frac{\rd v}{2\pi\ri},
    \end{split}
\end{equation}
where we used Lemma \ref{lm:telescope} and the assumption \eqref{eq:assumption_beta} in the last equation. The right-hand side of the above integral can be evaluated explicitly using residue computations
\begin{equation}
    \begin{split}
        \oint_\rC \ftn_m(v)\cdot \frac{1}{v-u} \cdot \prod_{i=\lbd}^{\lambda_m} \frac{u-\alpha_i}{v-\alpha_i}\frac{\rd v}{2\pi\ri}
        &=\oint_\rC \frac{1}{v-u} \cdot \frac{\prod_{i=\lbd}^{\lambda_m} (u-\alpha_i)}{\prod_{i=1}^m(v-\beta_i)}\frac{\rd v}{2\pi\ri}\\
        &=-\frac{\prod_{i=\lbd}^{\lambda_m} (u-\alpha_i)}{\prod_{i=1}^m(u-\beta_i)}
        =-\ftn_m(u).
    \end{split}
\end{equation}
This completes the proof.

%%%%%%%%%%%%%%%%%%%%%%%%%%%%%%%%%%%%%%%%%%%%%%%%%%%%%%%%%%%%
\section{Applications}
\label{sec:applications}

In this section, we present two more applications of Theorem \ref{prop:main}.
The first application is about the homogeneous version of Theorem \ref{thm:LPP_path-to-line}. Setting $\alpha_i=-1$ and $\beta_i=0$ for all $i$ in Theorem \ref{thm:LPP_path-to-line} gives a Sch\"utz-type determinantal formula for the density function in homogeneous DLPP from an arbitrary initial profile to a column; see \eqref{eq:Llambda-schutz}. We show how to use a known Fredholm determinant formula obtained in \cite{Liu22a} to derive \eqref{eq:Llambda-schutz}, using Theorem \ref{prop:main}. Thus, this gives an indirect proof of the homogeneous version of Theorem \ref{thm:LPP_path-to-line}. See Section \ref{subsec:equivalence_formulas_column} for this application.

The second application concerns the TASEP path-integral formula obtained in \cite{Matetski-Quastel-Remenik21}. In Section \ref{subsec:equal-time_path_integral}, we provide a new proof of this formula using  Theorem \ref{prop:main}.  We are particularly grateful to Yuchen Liao for suggesting this application and sharing preliminary calculations with us about why Theorem \ref{prop:main} can be applied for the TASEP equal-time multipoint distribution (which correspond to the material in Section \ref{subsec:equal-time_path_integral} up to Lemma \ref{lm:Taround0}).

%%%%%%%%%%%%%%%%%%%%%%%%%%%%%%%%%%%%%%
\subsection{Indirect proof of the path-to-line joint distribution in homogeneous DLPP}
\label{subsec:equivalence_formulas_column}

By setting $\alpha_i=-1$ and $\beta_i=0$ for all $i$ and evaluating the integral with respect to $x$, the homogeneous version of Theorem \ref{thm:LPP_path-to-line} yields the following result. For notational convenience, we let $N$ denote the column dimension in this application. This choice aligns with the conventions in \cite{Liu22a}, upon which our analysis heavily relies.

\begin{cor}[\cite{Hong-Liu-Tripathi26}]
Assume  $0\le t_1\le \cdots\le t_N$. Use the same notation as in Theorem \ref{thm:LPP_path-to-line}. We have
\begin{equation} \label{eq:Llambda-schutz}
\prob \left( \bigcap_{i=1}^N \left\{\ELPP_{\boldsymbol{\lambda}}(m, i) \le t_i\right\}\right) = 
\det\left[\oint_{|z|>1} z^{i-j-1}(z+1)^{\lambda_j-m}\mathrm{e}^{t_iz} \frac{\rd z}{2\pi\ri}\right]_{i,j=1}^N.
\end{equation}
\end{cor}

The main goal of this subsection is to provide an indirect proof of the above corollary using Theorem \ref{prop:main} and the following result from \cite{Liu22a}. Note that the original statement in \cite{Liu22a} is more general, concerning  the multipoint distribution of TASEP with a general initial condition. We translate it into the language of DLPP for the specific probability of interest in this application.

\begin{thm}[\cite{Liu22a}]
\label{thm:liu22mult}
Assume $m \ge \lambda_1+1$, and $0 \le t_1 \le t_2 \le \cdots \le t_N$. We have
\begin{equation}
\label{eq:liu22mult}
    \prob \left( \bigcap_{i=1}^N \left\{\ELPP_{\boldsymbol{\lambda}}(m, i) \le t_i\right\}\right)= \oint_0 \cdots \oint_0  \mathcal{D}(z_1,\ldots,z_{N-1}) \prod_{\ell=1}^{N-1} \frac{\rd z_\ell}{2\pi \ri z_\ell(1-z_\ell )}.
\end{equation}
where $\mathcal{D}(z_1,\ldots,z_{N-1})$ is defined in terms of series expansion in Definition \ref{defn:Dlambda}. Moreover, the notation $\oint_a$ denotes integration along a small circle around $z=a$, with the counterclockwise orientation.
\end{thm}

To introduce the formula for $\mathcal{D}(z_1,\ldots,z_{N-1})$ appearing in \eqref{eq:liu22mult}, we first introduce some notation.
Given $W=(w_1,\ldots,w_n)\in \mathbb{C}^n$ and $W'=(w_1',\ldots,w_m')\in \mathbb{C}^m$, we denote 
\begin{equation}
W\sqcup W':= (w_1,\ldots,w_n, w_1',\ldots,w_m')\in \mathbb{C}^{m+n}.
\end{equation}
Assume in addition that $n=m$ and $w_i\neq w_j'$ for all $1\leq i,j\leq n$. We denote 
\begin{equation}
\rC(W;W'):= \det\begin{bmatrix}
\displaystyle\frac{1}{w_i-w_j'}
\end{bmatrix}_{1\leq i,j\leq n} = (-1)^{\frac{n(n-1)}{2}}\frac{\prod_{1\leq i<j\leq n}(w_j-w_i)(w_j'-w_i')}{\prod_{1\leq i,j\leq n}(w_i-w_j')},
\end{equation}
which is the usual Cauchy determinant.

Consider the following two simply connected regions of $\complexC$:
\begin{equation}\label{eq:regions}
\Omega_\LL:=\left\{w\in\mathbb{C}: |w+1|<\frac{1}{2}-\epsilon\right\},\quad \Omega_\RR:=\left\{w\in\mathbb{C}: |w|<\frac{1}{2}-\epsilon\right\},
\end{equation}
where $\epsilon>0$ is a small constant.

Suppose $\Sigma_{N,\LL}^{\out},\ldots,\Sigma_{2,\LL}^{\out}$, $\Sigma_{1,\LL}$, $\Sigma_{2,\LL}^{\inn},\ldots,\Sigma_{N,\LL}^\inn$ are $2N-1$ nested simple closed contours, from outside to inside, in $\Omega_\LL$ enclosing the point $-1$. Similarly,
$\Sigma_{N,\RR}^{\out},\ldots,\Sigma_{2,\RR}^{\out}$, $\Sigma_{1,\RR}$, $\Sigma_{2,\RR}^{\inn},\ldots,\Sigma_{N,\RR}^\inn$ are $2N-1$ nested simple closed contours, from outside to inside, in $\Omega_\RR$ enclosing the point $0$.  See Figure~\ref{fig:contours_finite_time} for an illustration of the contours. These contours are all counterclockwise oriented. 

	\begin{figure}[t]
 	\centering
 	\begin{tikzpicture}[scale=1]
 	\draw [line width=0.4mm,lightgray] (-2,0)--(7,0) node [pos=1,right,black] {$\realR$};
 	\draw [line width=0.4mm,lightgray] (4.5,-2.5)--(4.5,2.5) node [pos=1,above,black] {$\ri\realR$};
 	\fill (2.5,0) circle[radius=1.5pt] node [below,shift={(0pt,-3pt)}] {$-\frac{1}{2}$};
    \fill (4.5,0) circle[radius=2.2pt] node [below,shift={(0pt,-8pt)}] {$0$};
    \fill (0.5,0) circle[radius=2.2pt] node [below,shift={(0pt,-8pt)}] {$-1$};
        \draw[dashed] (4.5,0) circle (50pt);
        \draw[dashed] (0.5,0) circle (50pt);
        \draw[blue, thick] (4.5,0) circle (30pt);
        \draw[black, thick] (4.5,0) circle (20pt);
        \draw[blue, thick] (4.5,0) circle (10pt);
        \draw[black, thick] (0.5,0) circle (30pt);
        \draw[blue, thick] (0.5,0) circle (20pt);
        \draw[black, thick] (0.5,0) circle (10pt);
                \node[text width=0.1cm,font=\bfseries] at (-0.8,1.2) {\scriptsize $\Omega_\LL$};
                 \node[text width=0.1cm,font=\bfseries] at (5.55,1.2) {\scriptsize $\Omega_\RR$};
			\node[text width=0.1cm,font=\bfseries, blue] at (4.5,1.3) {\scriptsize $\Sigma_{2,\RR}^\out$};
                \node[text width=0.1cm,font=\bfseries] at (3.9,0.7) {\scriptsize $\Sigma_{1,\RR}$};
                \node[text width=0.1cm,font=\bfseries, blue] at (4.6,0.35) {\scriptsize $\Sigma_{2,\RR}^\inn$};
			\node[text width=0.1cm,font=\bfseries] at (0.5,1.3) {\scriptsize $\Sigma_{2,\LL}^\out$};
            \node[text width=0.1cm,font=\bfseries, blue] at (0.7,0.7) {\scriptsize $\Sigma_{1,\LL}$};
                \node[text width=0.1cm,font=\bfseries] at (0.1,0.35) {\scriptsize $\Sigma_{2,\LL}^\inn$};
		\end{tikzpicture}

 	\caption{Illustration of the contours for $N=2$: The regions $\Omega_\LL$ and $\Omega_\RR$ are the interior of the two dashed circles, from left to right; the three contours around $-1$ from outside to inside are $\Sigma_{2,\LL}^{\out},\Sigma_{1,\LL},\Sigma_{2,\LL}^{\inn}$ respectively; the three contours around $0$ from outside to inside are $\Sigma_{2,\RR}^{\out},\Sigma_{1,\RR},\Sigma_{2,\RR}^{\inn}$ respectively.}    
    \label{fig:contours_finite_time}
\end{figure}	

\begin{defn}[Definition of $\mathcal{D} (z_1,\ldots,z_{N-1})$] \label{defn:Dlambda}
The series expansion of $\mathcal{D}(z_1,\ldots,z_{N-1})$ is given by
\begin{equation}
\label{eq:mcDz1zN-1serexp}
    \mathcal{D} (z_1,\ldots,z_{N-1}) = \sum_{\bn \in (\intZ_{\ge0})^N} \frac{(-1)^{|\bn|}}{(\bn!)^2} \mathcal{D}_{\bn }(z_1,\ldots,z_{N-1})
\end{equation}
where
\begin{equation}
\label{eq:mcDnlambda}
\begin{split}
    \mathcal{D}_{\bn }&=\mathcal{D}_{\bn }(z_1,\ldots,z_{N-1}) \\
    &= \prod_{i=1}^{N-1} (1-z_i)^{n_i}(1-z_i^{-1})^{n_{i+1}} \\
    & \quad \cdot \prod_{i=2}^{N} \prod_{\ell_i=1}^{n_i} \left(\frac{1}{1-z_{i-1}} \int_{\Sigma^\inn_{i,\LL}} \ddbar{u^{(i)}_{\ell_i}}{} - \frac{z_{i-1}}{1-z_{i-1}} \int_{\Sigma^\out_{i,\LL}} \ddbar{u^{(i)}_{\ell_i}}{}  \right) \prod_{\ell_1=1}^{n_1} \int_{\Sigma_{1,\LL}} \ddbar{u^{(1)}_{\ell_1}}{} \\
    &\quad \cdot  \prod_{i=2}^{N} \prod_{\ell_i=1}^{n_i} \left(\frac{1}{1-z_{i-1}} \int_{\Sigma^\inn_{i,\RR}} \ddbar{v^{(i)}_{\ell_i}}{} - \frac{z_{i-1}}{1-z_{i-1}} \int_{\Sigma^\out_{i,\RR}} \ddbar{v^{(i)}_{\ell_i}}{}  \right) \prod_{\ell_1=1}^{n_1} \int_{\Sigma_{1,\RR}} \ddbar{v^{(1)}_{\ell_1}}{} \\
    &\quad \cdot \prod_{i=1}^{N} \prod_{\ell_i=1}^{n_i} \frac{G_i(u^{(i)}_{\ell_i})}{G_i(v^{(i)}_{\ell_i})} \cdot \det \left[ \ch (v^{(1)}_i, u^{(1)}_j) \right]_{i,j=1}^{n_1} \cdot \prod_{i=1}^{N-1}\rC(U^{(i)}\sqcup V^{(i+1)}; V^{(i)}\sqcup U^{(i+1)}) \cdot \rC(U^{(N)}; V^{(N)}),
\end{split}
\end{equation}
where $\bn = (n_1, \ldots, n_N)$, $|\bn| := n_1+\cdots+n_N$, and $\bn ! := n_1! \cdots n_N!$. The vectors $U^{(\ell)} = (u^{(\ell)}_1, \ldots, u^{(\ell)}_{n_\ell})$ and $V^{(\ell)} = (v^{(\ell)}_1, \ldots, v^{(\ell)}_{n_\ell})$ for $1 \le \ell \le N$. $\ch(v, u)$ is the characteristic function defined in \eqref{eq:homogeneous_characteristic_function} with $\lbd=0$ and $p=-1/2$, and the functions $G_i$ are given by
\begin{equation}
\label{eq:defG}
G_i(w) = \begin{cases}
w(w+1)^{-m-1}e^{t_1w},  &  \text{if $i=1$}, \\
we^{(t_i-t_{i-1})w},  &  \text{if $i\ge 2$.}
\end{cases}
\end{equation}
\end{defn}

Note that the formula for the function $G_i$ is remarkably simple. This simplicity arises directly from the column structure of the multipoint distribution considered in Theorem \ref{thm:liu22mult}. Consequently, we can simplify the multipoint distributions in \eqref{eq:liu22mult}. The above formula for $\mathcal{D}(z_1,\ldots,z_{N-1})$ is expressed as an infinite sum. However, it is not difficult to show that when $n_\ell>N$, the integral on the right-hand side of \eqref{eq:mcDnlambda} is zero. For a proof of this fact, we refer the reader to \cite[Remark 2.8]{Liu22a}. Now, recalling \eqref{thm:liu22mult} and interchanging the order of integration and summation, we obtain
\begin{equation}
\label{eq:LPPsimple}
\begin{split}
\prob \left( \bigcap_{i=1}^N \left\{\ELPP_{\boldsymbol{\lambda}}(m, i) \le t_i\right\}\right) = \sum_{\bn \in (\intZ_{\ge0})^N} \frac{(-1)^{|\bn|}}{(\bn!)^2} \mathrm{D}_{\bn },
\end{split}
\end{equation}
where
\begin{equation}
\mathrm{D}_{\bn }:= \oint_0 \cdots \oint_0  \mathcal{D}_{\bn}(z_1,\ldots,z_{N-1}) \prod_{\ell=1}^{N-1} \frac{\rd z_\ell}{2\pi \ri z_\ell(1-z_\ell )}.
\end{equation}

Below we show that the $\mathrm{D}_\bn$ function can be simplified. We remark that a similar simplification strategy was developed in \cite{Liu-Ortiz25}, where the authors considered the integrals in the multipoint distribution formula for the KPZ fixed point obtained in \cite{Liu22a} when all the time parameters are equal. In fact, the structure of our function $\mathrm{D}_{\bn}$ is similar to that of the function $\hat{\mathcal{D}}_{\bn}$ in \cite{Liu-Ortiz25}. Moreover, the statements and proofs of Lemmas \ref{lm:simplifiedcolmult} and \ref{lm:Dnsimplisimpli} are similar to those of Lemmas 2.3 and 2.4, respectively, in \cite{Liu-Ortiz25}. However, there are still differences in the integration contours and the functions involved, which lead to different restrictions on indices that make nontrivial contributions.  Below we include the details of the arguments that are affected by these differences and omit the parts similar to \cite{Liu-Ortiz25}.

\medskip

We first evaluate the $z$-integrals in \eqref{eq:liu22mult} by simplifying the $u$-integrals in $\mathrm{D}_{\bn }$, which yields a more compact formula stated in Lemma \ref{lm:simplifiedcolmult} below.

For notational convenience, we set
\begin{equation}
\Sigma^\out_{1,\LL} = \Sigma_{1,\LL}, \quad \Sigma^\inn_{1,\RR} = \Sigma_{1,\RR},
\end{equation}
throughout this subsection.

\begin{lm}
\label{lm:simplifiedcolmult}
Under the same assumptions as in Theorem \ref{thm:liu22mult}, 
\begin{enumerate} [(i)] 
\item 
we have
\begin{equation}
\label{eq:rDnlambda}
\begin{split}
\mathrm{D}_{\bn } &= \prod_{i=1}^N \prod_{\ell_i=1}^{n_i} \int_{\Sigma^\out_{i,\LL}} \ddbar{u^{(i)}_{\ell_i}}{}\int_{\Sigma^\inn_{i,\RR}} \ddbar{v^{(i)}_{\ell_i}}{} \prod_{i=1}^{N} \prod_{\ell_i=1}^{n_i} \frac{G_i(u^{(i)}_{\ell_i})}{G_i(v^{(i)}_{\ell_i})}\\
&\quad \cdot \det \left[ \ch (v^{(1)}_i, u^{(1)}_j) \right]_{i,j=1}^{n_1} \cdot \prod_{i=1}^{N-1}\rC(U^{(i)}\sqcup V^{(i+1)}; V^{(i)}\sqcup U^{(i+1)}) \cdot \rC(U^{(N)}; V^{(N)}),
\end{split}
\end{equation}
where $G_i$ is defined in \eqref{eq:defG}.

\item if $k_i :=n_i - n_{i+1} \in \{0,1\}$ for all $1 \le i \le N$, with the convention that $n_{N+1} = 0$, we have
\begin{equation}
\label{eq:rDnlambdasimplified1}
\begin{split}
\mathrm{D}_{\bn } 
&= \prod_{i=1}^{N-1}\frac{n_i!}{k_i!}\prod_{i=1}^N \prod_{\hat{\ell}_i=1}^{k_i} \int_{\Sigma_{1,\LL}} \ddbar{\hat{u}^{(i)}_{\hat{\ell}_i}}{}\prod_{i=1}^N \prod_{\ell_i=1}^{n_i} \int_{\Sigma^\inn_{i,\RR}} \ddbar{v^{(i)}_{\ell_i}}{} \\
& \quad \cdot  \det \left[ \ch(v^{(1)}_i, \hat{u}_j) \right]_{i,j=1}^{n_1}  \cdot \prod_{i=1}^{N}\rC(V^{(i+1)} \sqcup \hat{U}^{(i)}; V^{(i)}) \cdot \prod_{i=1}^{N}  \frac{\prod_{\hat{\ell}_i=1}^{k_i}g_i(\hat{u}^{(i)}_{\hat{\ell}_i})}{\prod_{\ell_i=1}^{n_i}G_i(v^{(i)}_{\ell_i})}, 
\end{split}
\end{equation}
where $\hat{U}^{(i)} = (\hat{u}^{(i)}_1, \ldots, \hat{u}^{(i)}_{k_i})$ for $1 \le i \le N$, $\hat{U} = \hat{U}^{(N)} \sqcup  \cdots \sqcup \hat{U}^{(1)} = (\hat{u}_1,\ldots, \hat{u}_{n_1})$, and $G_i$ is defined in \eqref{eq:defG}. The functions $g_i$ are defined as
\begin{equation}
\label{eq:defg}
    g_i(w) = \prod_{k=1}^i G_k(w) = w^i (w+1)^{-m-1}e^{t_iw}. 
\end{equation}
Moreover, for all other $\bn$ we have $\mathrm{D}_{\bn } = 0$.

\end{enumerate}
\end{lm}

\begin{proof}

The proof is essentially the same as that of \cite[Lemma 2.3]{Liu-Ortiz25}, except that here we work with the $u$-variables instead. Note that the contours in \cite{Liu-Ortiz25} are infinite, but they have a nesting structure (with respect to $-\infty$ for the left contours and $+\infty$ for the right contours) similar to that of the contours in \eqref{eq:mcDnlambda}.

\begin{enumerate} [(i)]
\item 
We expand $u$-integrals in \eqref{eq:mcDnlambda} by only taking one contour for each variable in one term. Consider one such integral. If any of the $u^{(i)}_{\ell_i}$ contours is $\Sigma^{\inn}_{i,\LL}$, then consider the integral corresponding to the largest such $i$. Assume $i \ge 2$. Since the only possible poles of $u^{(i)}_{\ell_i}$ are $v^{(j)}_{\ell_j}$'s in the right half plane, $u^{(i-1)}_{\ell_{i-1}} \in \Sigma^{\inn}_{i-1,\LL} \cup \Sigma^{\out}_{i-1,\LL}$, and $u^{(i+1)}_{\ell_{i+1}} \in \Sigma^{\out}_{i+1,\LL}$, the integrand, viewed as a function of $u^{(i)}_{\ell_i}$, is analytic inside $\Sigma^{\inn}_{i,\LL}$. Hence, the integral is zero. Therefore, the integral vanishes if any $\Sigma^{\inn}_{i,\LL}$ is included in the term from expansion of \eqref{eq:mcDnlambda}. Now we drop all these vanishing terms and have only one surviving term
\begin{equation*}
\mathcal{D}_{\bn} = \prod_{i=1}^{N-1} (1-z_i)^{n_i}\prod_{i=2}^{N} \prod_{\ell_i=1}^{n_i} \left(\frac{1}{1-z_{i-1}} \int_{\Sigma^\inn_{i,\RR}} \ddbar{v^{(i)}_{\ell_i}}{} - \frac{z_{i-1}}{1-z_{i-1}} \int_{\Sigma^\out_{i,\RR}} \ddbar{v^{(i)}_{\ell_i}}{}  \right) \int_{\Sigma^\out_{i,\LL}} \ddbar{u^{(i)}_{\ell_i}}{} \cdots,
\end{equation*}
where we suppressed the integrand, which is not relevant to the $z$-variables. The function above is analytic in $z_i$ at the origin for each $i$. Hence, evaluating the $z-$integrals $\oint_0 \cdots \oint_0  \mathcal{D}_{\bn} \prod_{\ell=1}^{N-1} \frac{\rd z_\ell}{2\pi \ri z_\ell(1-z_\ell )}$ is the same as inserting $z_i=0$ in the expression above. This yields \eqref{eq:rDnlambda}.

\item 
To prove (ii), we start with the formula \eqref{eq:rDnlambda} and evaluate the $u$-integrals successively. The proof is basically the same as that of \cite[Lemma 2.3(ii)]{Liu-Ortiz25}, and it yields \eqref{eq:rDnlambdasimplified1} together with
$n_1 \ge n_2 \ge \cdots \ge n_N$. We omit this part of the proof.
It remains to establish the additional fact that $n_i \le n_{i+1}+1$ for all $1 \le i \le N$, which we prove next.

Consider the formula \eqref{eq:rDnlambda}. We first evaluate the integrals with respect to $v^{(1)}_{\ell_1}$ along the contour $\Sigma^{\inn}_{1,\RR}$. Note that the only poles of $v^{(1)}_{\ell_1}$ inside $\Sigma^{\inn}_{1,\RR}$ are the simple poles at $v^{(1)}_{\ell_1}=0$ and at $v^{(1)}_{\ell_1}=v^{(2)}_{\ell_2}$, arising from the Cauchy determinant $\rC(U^{(1)} \sqcup V^{(2)}; V^{(1)} \sqcup U^{(2)})$. Moreover, this Cauchy determinant vanishes if two distinct $v^{(1)}_{\ell_1}$ variables are evaluated at the same pole. Since there are $n_1$ variables $v^{(1)}_{\ell_1}$ and only $n_2+1$ possible poles, namely $0$ and $v^{(2)}_{\ell_2}$ for $1\leq \ell_2\leq n_2$, it follows that $\mathrm{D}_{\bn}=0$ unless $n_1\leq n_2+1$. Continuing the same argument successively for the variables $v^{(i)}_{\ell_i}$, $i=2,\ldots,N$, we find that $\mathrm{D}_{\bn}=0$ unless $n_i\le n_{i+1}+1$ for all $1\leq i\leq N$. Combining this with the fact that $n_1 \ge n_2 \ge \cdots \ge n_N$, we conclude that $\mathrm{D}_{\bn}=0$ unless $k_i=n_i-n_{i+1}\in \{0,1\}$.
\end{enumerate}
\end{proof}

The next step is to further simplify \eqref{eq:rDnlambdasimplified1} by reorganizing the $v$-integrals. We assume that $n_i-n_{i+1}\in \{0,1\}$ for each $1 \leq i \leq N$, since otherwise $\mathrm{D}_{\bn}=0$, as shown in Lemma \ref{lm:simplifiedcolmult} (ii).

\begin{lm}
\label{lm:Dnsimplisimpli}
Assume $k_i = n_i - n_{i+1} \in \{0,1\}$ for each $i=1,\ldots,N$ where we set $n_{N+1} = 0$ for convenience. We have
\begin{equation}
\label{eq:rDnlambdasimplified2}
\begin{split}
    &\mathrm{D}_{\bn} \\
    &= \left (\prod_{i=1}^{N-1}\frac{n_i!}{k_i!} \right)^2 \prod_{i=1}^N \prod_{\hat{\ell}_i=1}^{k_i} \int_{\Sigma_{1,\LL}} \ddbar{\hat{u}^{(i)}_{\hat{\ell}_i}}{} \int_{\Sigma_{1,\RR}} \ddbar{\hat{v}^{(i)}_{\hat{\ell}_i}}{}   \det \left[ \ch(\hat{v}_i, \hat{u}_j) \right]_{i,j=1}^{n_1}  \prod_{i=1}^N \det \left[ h_i (\hat{u}^{(i)}_a, \hat{v}^{(i)}_b) \right]_{a,b=1}^{k_i}  \prod_{i=1}^{N}  \prod_{\hat{\ell}_i=1}^{k_i}g_i(\hat{u}^{(i)}_{\hat{\ell}_i}),
\end{split}
\end{equation}
where $\hat{U}^{(i)} = (\hat{u}^{(i)}_1, \ldots, \hat{u}^{(i)}_{k_i})$, $\hat{V}^{(i)} = (\hat{v}^{(i)}_1, \ldots, \hat{v}^{(i)}_{k_i})$ and $\hat{U} = \hat{U}^{(N)} \sqcup  \cdots \sqcup \hat{U}^{(1)} = (\hat{u}_1,\ldots, \hat{u}_{n_1})$, $\hat{V} = \hat{V}^{(N)} \sqcup  \cdots \sqcup \hat{V}^{(1)} = (\hat{v}_1,\ldots, \hat{v}_{n_1})$ and $g_i$ is defined in \eqref{eq:defg}. The functions $h_i(u,v)$ are defined for all $(u,v) \in \Sigma_{1,\LL} \times \Sigma_{1,\RR}$ as follows
\begin{equation}
\label{eq:defh}
h_i(u,v) = 
\begin{cases}
\displaystyle \frac{1}{G_1(v)(u-v)}, & i=1, \\
\displaystyle \prod_{j=2}^i \int_{\Sigma^\inn_{j,\RR}} \ddbar{v_{j}}{} \frac{1}{(v_2-v) \cdot \prod_{j=2}^{i-1}(v_{j+1} - v_{j}) \cdot (u-v_i)} \cdot \frac{1}{\prod_{j=2}^i G_j(v_j) \cdot G_1(v)}, & 2 \le i \le N,
\end{cases}
\end{equation}
where $G_i$ is defined in \eqref{eq:defG}.
\end{lm}
\begin{proof}
The proof is essentially the same as the proof of  \cite[Lemma 2.4]{Liu-Ortiz25}, except that here we reorganize the $v$-integrals rather than the $u$-integrals. We therefore omit the proof.
\end{proof}

Next, we simplify the formula \eqref{eq:defh} for $h_i$. Note that, in this formula, $u$, $v$, and $v_i$ can be chosen so that $|u| > |v| > |v_i|$. Therefore, we may write
\begin{equation}
\label{eq:simplified_hi}
h_i(u,v) = \sum_{\ell=0}^\infty \frac{1}{u^{\ell+1}} \tilde h^{(\ell)}_i(v)
\end{equation}
where
\begin{equation}
\label{eq:tilde_hi}
\tilde h^{(\ell)}_i(v) := 
\begin{dcases}
\frac{v^{\ell}}{G_1(v)}, & i=1, \\
\prod_{j=2}^i \int_{\Sigma^\inn_{j,\RR}} \ddbar{v_{j}}{} \frac{v^\ell_i}{(v_2-v) \cdot \prod_{j=2}^{i-1}(v_{j+1} - v_{j})} \cdot \frac{1}{\prod_{j=2}^i G_j(v_j) \cdot G_1(v)}, & 2 \le i \le N.
\end{dcases}
\end{equation}
Recall the formula \eqref{eq:defG} for $G_i$. For $i \ge 2$ in \eqref{eq:tilde_hi}, note that the integrand, viewed as a function of $v_i$, is analytic inside $\Sigma^\inn_{i,\RR}$ unless $\ell=0$.
Therefore, $\tilde h^{(\ell)}_i(v) = 0$ for $2 \le i \le N$ and $\ell \ge 1$, and hence
\begin{equation}
\label{eq:2simplified_hi}
h_i(u,v) = \1_{i=1}\sum_{\ell=0}^\infty \frac{v^\ell}{u^{\ell+1}} \frac{1}{G_1(v)} + \1_{i \ne 1}\frac{\tilde h^{(0)}_i(v) }{u}.
\end{equation}

Now we need the following generalized Andreief's identity from \cite{Liu-Ortiz25}.

\begin{lm}[Lemma 1.2 of \cite{Liu-Ortiz25}]
\label{lm:genAndreiefId}
Let $I_1,\ldots, I_m$ be a partition of $\{1,\ldots,n\}$. Let $X\subset \complexC$ be a measurable set and $\mu$ be a measure on $X$. Suppose $A_i(x)$ and $B_i(x)$, $1 \leq i \leq n$, are two sequences of functions on $X$ such that $A_i(x)B_j(x)$ is integrable with respect to $\rd\mu$ for all $1 \leq i,j \leq n$. Then we have
\begin{equation}
\label{eq:genAndreiefId}
\det \left[\int_X A_i(x)B_j(x) \rd \mu(x) \right]_{i,j=1}^n = \frac{1}{\prod_{k=1}^m |I_k|!} \int_{X^n} \det \left[ A_i(x_j) \right]_{i,j=1}^n \prod_{k=1}^m \det \left[B_i(x_j) \right]_{i,j\in I_k} \prod_{i=1}^n \rd \mu (x_i).
\end{equation}
\end{lm}

We apply \eqref{eq:genAndreiefId} to the $\hat u$-variables in \eqref{eq:rDnlambdasimplified2} to rewrite $\mathrm{D}_{\bn}$ as follows. When $k_i = n_i-n_{i+1} \in \{0,1\}$, $1 \le i \le N$, we obtain
\begin{equation}
\label{eq:rDnlambdasimplified3}
\begin{split}
    \mathrm{D}_{\bn} 
    &= \prod_{i=1}^{N}\frac{(n_i!)^2}{k_i!}  \prod_{i=1}^N \prod_{\hat{\ell}_i=1}^{k_i} \int_{\Sigma_{1,\RR}}  \ddbar{\hat{v}^{(i)}_{\hat{\ell}_i}}{}  \det \left[\int_{\Sigma_{1,\LL}} h_i \left(u, \hat{v}^{(i)}_{\hat \ell_i}\right)  g_i(u) \, \ch\left(\hat{v}^{(j)}_{\hat \ell_j}, u\right) \ddbar{u}{}  \right] _{(i,\hat \ell_i), (j, \hat \ell_j)},
\end{split}
\end{equation}
where the row/column indices are chosen from $\{(i,\hat\ell_i) : 1 \le i \le N, 1 \le \hat \ell_i \le k_i\}$. Recall that, by Lemma \ref{lm:simplifiedcolmult} (ii), $\mathrm{D}_{\bn}=0$ for all other $\bn$. Then, using \eqref{eq:LPPsimple}, \eqref{eq:rDnlambdasimplified3} and noting $|\bn| = \sum _{1\le i\le N} ik_i$, we find that
\begin{equation}
\begin{split}
&\prob \left( \bigcap_{i=1}^N \left\{\ELPP_{\boldsymbol{\lambda}}(m, i) \le t_i\right\}\right)\\
&= \sum_{k_1,\ldots,k_N \in \{0,1\}} (-1)^{\sum _{1\le i\le N} ik_i}\prod_{i=1}^N \prod_{\hat{\ell}_i=1}^{k_i} \int_{\Sigma_{1,\RR}}  \ddbar{\hat{v}^{(i)}_{\hat{\ell}_i}}{}  \det \left[\int_{\Sigma_{1,\LL}} h_i \left(u, \hat{v}^{(i)}_{\hat \ell_i}\right)  g_i(u) \, \ch\left(\hat{v}^{(j)}_{\hat \ell_j}, u\right) \ddbar{u}{}  \right] _{(i,\hat \ell_i), (j, \hat \ell_j)} \\
&= \sum_{\substack{I \subseteq  \{1,\ldots,N \} \\ \text{$I$ is increasing}}} (-1)^{\sum _{i\in I} i}\prod_{i \in I} \int_{\Sigma_{1,\RR}}  \ddbar{\hat v_i}{}  \det \left[\int_{\Sigma_{1,\LL}} h_a \left(u, \hat v_a\right)  g_a(u) \, \ch\left(\hat v_b, u\right) \ddbar{u}{}  \right] _{(a,b) \in I \times I}.
\end{split}
\end{equation}
For notational convenience, define
\begin{equation}
\label{eq:defrS}
\mathrm{S}_i = \begin{dcases}
    \{0,1,\ldots\} , & i = 1, \\
    \{0 \}, & 2 \le i \le N. 
\end{dcases}
\end{equation} 
Therefore, using \eqref{eq:2simplified_hi}, we obtain 
\begin{equation}
\begin{split}
&\prob \left( \bigcap_{i=1}^N \left\{\ELPP_{\boldsymbol{\lambda}}(m, i) \le t_i\right\}\right)\\
&= \sum_{\substack{I \subseteq  \{1,\ldots,N \} \\ \text{$I$ is increasing}}} (-1)^{\sum _{i\in I} i}\prod_{i \in I} \int_{\Sigma_{1,\RR}}  \ddbar{\hat v_i}{} \det \left[\sum_{\ell_a \in \mathrm{S}_a} \tilde h_a^{(\ell_a)} \left(\hat v_a\right) \int_{\Sigma_{1,\LL}} \frac{g_a(u)}{u^{\ell_a+1}}\,  \ch\left(\hat v_b, u\right) \ddbar{u}{}  \right] _{(a,b) \in I \times I} \\
&= \sum_{\substack{I \subseteq  \{1,\ldots,N \} \\ \text{$I$ is increasing}}} (-1)^{\sum _{i\in I} i} \prod_{i \in I} \int_{\Sigma_{1,\RR}}  \ddbar{\hat v_i}{}  \sum_{\ell_i \in \mathrm{S}_i, \, i \in I} \det \left[\tilde h_a^{(\ell_a)} \left(\hat v_a\right) \int_{\Sigma_{1,\LL}} \frac{g_a(u)}{u^{\ell_a+1}} \, \ch\left(\hat v_b, u\right) \ddbar{u}{}  \right] _{(a,b) \in I \times I} \\
&= \sum_{\substack{I \subseteq  \{1,\ldots,N \} \\ \text{$I$ is increasing}}} (-1)^{\sum _{i\in I} i} \prod_{i \in I} \int_{\Sigma_{1,\RR}}  \ddbar{\hat v_i}{}  \sum_{\ell_i \in \mathrm{S}_i, \, i \in I}\det \left[\tilde h_b^{(\ell_b)} \left(\hat v_b\right) \int_{\Sigma_{1,\LL}} \frac{g_a(u)}{u^{\ell_a+1}} \, \ch\left(\hat v_b, u\right) \ddbar{u}{}  \right] _{(a,b) \in I \times I} .
\end{split}
\end{equation}
Now we interchange the integral and the summation. Recalling \eqref{eq:defrS}, note that this interchange only needs to be justified in the case $i=1$. This follows from the dominated convergence theorem, since $\ch(\hat v_1,u)$ is analytic at $\hat v_1=0$, $\tilde h^{(\ell_1)}_1(\hat v_1)=O(\hat v_1^{\ell_1-1})$ near $\hat v_1=0$, and $\Sigma_{1,\RR}$ can be chosen to be a sufficiently small circle around $0$.
Then
\begin{equation*}
\begin{split}
&\prob \left( \bigcap_{i=1}^N \left\{\ELPP_{\boldsymbol{\lambda}}(m, i) \le t_i\right\}\right)\\
&= \sum_{\substack{I \subseteq  \{1,\ldots,N \} \\ \text{$I$ is increasing}}} (-1)^{\sum _{i\in I} i}   \sum_{\substack{\ell_i \in \mathrm{S}_i, \\ i \in I}} \prod_{i \in I} \int_{\Sigma_{1,\RR}}  \ddbar{\hat v_i}{} \det \left[\tilde h_b^{(\ell_b)} \left(\hat v_b\right) \int_{\Sigma_{1,\LL}} \frac{g_a(u)}{u^{\ell_a+1}} \, \ch\left(\hat v_b, u\right) \ddbar{u}{}  \right] _{(a,b) \in I \times I} \\
&= \sum_{\substack{I \subseteq  \{1,\ldots,N \} \\ \text{$I$ is increasing}}} (-1)^{\sum _{i\in I} i}\sum_{\substack{\ell_i \in \mathrm{S}_i, \\ i \in I}}  \det \left[ \int_{\Sigma_{1,\RR}}  \ddbar{\hat v}{} \tilde h_b^{(\ell_b)} \left(\hat v\right) \int_{\Sigma_{1,\LL}} \frac{g_a(u)}{u^{\ell_a+1}} \, \ch\left(\hat v, u\right) \ddbar{u}{}  \right] _{(a,b) \in I \times I} .
\end{split}
\end{equation*}
When $b = 1$ and $\ell_b \ge 1$, the integrand is analytic at $\hat v = 0$, and hence the integral vanishes. Therefore, the determinant is nonzero only when $\ell_i=0$ for all $i$, and hence
\begin{equation}
\begin{split}
&\prob \left( \bigcap_{i=1}^N \left\{\ELPP_{\boldsymbol{\lambda}}(m, i) \le t_i\right\}\right) \\&= \sum_{\substack{I \subseteq  \{1,\ldots,N \} \\ \text{$I$ is increasing}}}  \det \left[ (-1)^b \int_{\Sigma_{1,\RR}}  \ddbar{\hat v}{} \tilde h_b^{(0)} \left(\hat v\right) \int_{\Sigma_{1,\LL}} \frac{g_a(u)}{u} \, \ch\left(\hat v, u\right) \ddbar{u}{}  \right] _{(a,b) \in I \times I} \\
&= \det(\text{I}_N + M),
\end{split}
\end{equation}
where 
\begin{equation*}
\begin{split}
M_{ab} = -\int_{\Sigma_{1,\LL}} \ddbar{u}{} \prod_{j=1}^b \int_{\Sigma^\inn_{j,\RR}} \ddbar{v_{j}}{} \frac{1}{ \prod_{j=1}^{b-1}(v_{j} - v_{j+1})} \cdot \frac{1}{\prod_{j=1}^b G_j(v_j) } \cdot \frac{g_a(u)}{u} \cdot \ch\left(v_1, u\right),
\end{split}
\end{equation*}
where we have set $\hat v = v_1$ for notational convenience and used \eqref{eq:tilde_hi}.

A simple conjugation gives
\begin{equation} \label{eq:simplified_Llambda}
\begin{split}
\prob \left( \bigcap_{i=1}^N \left\{\ELPP_{\boldsymbol{\lambda}}(m, i) \le t_i\right\}\right)  = \det(\rI + K),
\end{split}
\end{equation}
where 
\begin{equation*}
K(u_1,u_2) = -\sum_{\ell=1}^N \frac{g_\ell(u_2)}{u_2}\prod_{j=1}^\ell \int_{\Sigma^\inn_{j,\RR}} \ddbar{v_{j}}{} \frac{1}{ \prod_{j=1}^{\ell-1}(v_{j} - v_{j+1})} \cdot \frac{1}{\prod_{j=1}^\ell G_j(v_j) }  \cdot \ch\left(v_1, u\right).
\end{equation*}

Set $\beta_i=0$, $p_i(v) = (v+1)^{-m-1} e^{t_i v}$, $q_i(u) = \frac{g_i(u)}{u} = u^{i-1}(u+1)^{-m-1} e^{t_i u}$, $f_i(v) = (v+1)^{\lambda_i+1}$, $g_i(u) = u^{-i}(u+1)^{\lambda_i+1}$ for all $1 \le i \le N$, and $H(v,u) := -\ch(v,u)$ in Theorem \ref{prop:main}. Note that this $H(v,u)$ satisfies \eqref{eq:fgHrel}, where we use Theorem \ref{thm:hitting} together with \eqref{eq:def_ftn}.
We find that \eqref{eq:simplified_Llambda} agrees with \eqref{eq:Llambda-schutz}. Hence, the formulas \eqref{eq:Llambda-schutz} and \eqref{eq:liu22mult} are equivalent.

%%%%%%%%%%%%%%%%%%%%%%%%%%%%%%%%%%%%%%%%%%%%%%%
\subsection{Equal-time multipoint distribution formula for TASEP}
\label{subsec:equal-time_path_integral}
The totally asymmetric simple exclusion process (TASEP) is an interacting particle system on the infinite lattice $\intZ$. Each site can contain at most one particle. The system evolves as follows. Each particle is equipped with an independent exponential clock of rate $1$. When a particle's clock rings, the particle attempts to jump to its right neighboring site. The jump is performed if that site is unoccupied; otherwise, the particle remains at its current site.
After each attempt, the clock is reset.

Set
\begin{equation*}
\mathcal{X}_N := \{(x_1, \ldots, x_N) \in \intZ^N : x_1 > \cdots >x_N \}.
\end{equation*}
Let $\rX(t) = (\rx_1(t), \ldots, \rx_N(t)) \in \mathcal{X}_N$ denote the configuration of the TASEP with $N$ particles at time $t \ge 0$, where the particles are labeled from right to left. We recall Sch\"utz's formula for the transition probability of the TASEP. For notational convenience, we use  $\prob_Y\left( \cdots \right) := \prob\left(\cdots  \mid  \rX(0) = Y\right)$. 
\begin{thm}[\cite{Sch97}]
For $X = (x_1,\ldots, x_N)$, $Y = (y_1,\ldots,y_N)$ in $\mathcal{X}_N$, and $t > 0$, the transition probability of TASEP is given by
\begin{equation}
\prob_Y\left( \rX(t) = X \right)
= \det \left[ \oint_{|w|>1} w^{i-j}(w+1)^{-x_i + y_j + j - i - 1} e^{tw}\, \ddbar{w}{} \right]_{i,j=1}^N.
\end{equation}
\end{thm}
Using Sch\"utz's formula, one can simply derive the following determinantal expression for the equal-time multipoint distribution by performing successive summations. We state the following lemma:

\begin{lm}
\label{lm:determinantal_mult}
Let $Y = (y_1,\ldots,y_N) \in \mathcal{X}_N$, let $m\ge 1$, and let $1 \le k_1 < \cdots < k_{m-1} < k_m = N$ be integers. Moreover, let $a_1,\ldots,a_m$ be integers such that $a_{i-1} + k_{i-1} \ge a_i + k_i$ for all $2 \le i \le m$. Then
\begin{equation}
\label{eq:equal-timemultinAi}
\prob_Y\left( \bigcap_{\ell=1}^m \{\rx_{k_\ell}(t) \ge a_\ell\}\right) = \det \left[ \oint_{|w|>1}  w^{i-j-1}(w+1)^{-A_i + y_j +j}e^{tw} \ddbar{w}{} \right]_{i,j=1}^N,
\end{equation}
where 
\begin{equation}
\label{eq:defAi}
   A_i := a_{\ell} + k_\ell, \quad  \text{and $1 \le \ell \le m$ is such that }{k_{\ell-1} < i \le k_\ell},
\end{equation}
with $k_0:=0$.
\end{lm}
\begin{proof}
Let $B = (b_1,\ldots, b_N) \in \mathcal{X}_ N$. We first compute
\begin{equation*}
\begin{split}
\prob_Y\left( \bigcap_{k=1}^N \{\rx_k(t) \ge b_k\}\right)
&= \sum_{x_1=b_1}^\infty \sum_{x_2=b_2}^{x_1-1} \cdots \sum_{x_{N-1} = b_{N-1}}^{x_{N-2}-1} \sum_{x_N = b_N}^{x_{N-1}-1}  \prob_Y(\rX(t) = X) \\
&=\det \left[ \oint_{>1}  w^{i-j-1}(w+1)^{-b_i + y_j +j - i}e^{tw} \ddbar{w}{} \right]_{i,j=1}^N,
\end{split}
\end{equation*}
where we used the linearity of the determinant. Set $b_i = A_i -i$ for $1 \le i \le N$, where $A_i$ is defined in \eqref{eq:defAi}.
Then, we have
\begin{equation*}
 \prob_Y\left( \bigcap_{k=1}^N \{\rx_k(t) \ge b_k\}\right) =    \prob_Y\left( \bigcap_{\ell=1}^m \{\rx_{k_\ell}(t) \ge a_\ell\}\right).
\end{equation*}
Hence, the lemma is proved.
\end{proof}

Recall the definition of the characteristic function in \eqref{eq:homogeneous_characteristic_function}. Since we are considering the particle locations of TASEP, we rewrite it in terms of the initial condition $Y=(y_1,\ldots,y_N)\in\mathcal{X}_N$. Let $(\hat G_k)_{k\ge 0}$ be a geometric random walk with transition probability given by $\prob(\hat G_{k+1} = x | \hat G_k=y) := \frac{1}{2^{y-x}} \1_{x<y}$, and $\tau$ is the hitting time of $\hat G$ to the strict epigraph of $Y$, namely $\tau := \min\{m\ge 0 : \hat G_m > y_{m+1}\}$.
Then  $\ch$ is the TASEP characteristic function of the initial condition $Y$ (see \cite{Liao-Liu25} and the discussions after \eqref{eq:homogeneous_characteristic_function}) 
\begin{equation}
\label{eq:chYprobrep}
    \ch(v,u) := \sum_{z\in \intZ} (2u+2)^z \cdot \mathbb{E}_{\hat G_0 = z} \left[\frac{2}{(2v+2)^{\hat G_{\tau}+1}}\cdot \left(\frac{-v}{v+1}\right)^{\tau} \1_{\tau < N} \right].
\end{equation}

We apply Theorem \ref{prop:main} to obtain the following:
\begin{lm} 
\label{lm:Taround0}
Assume the same notation and conventions in Lemma \ref{lm:determinantal_mult}. Let $\Sigma_{-1}$ be a small circle centered at $-1$. Define the functions $\rF_i$ and $\rf_i$ by
\begin{equation}
\label{eq:defrFrf}
\rF_i(w) = \frac{\rf_i(w) }{\rf_{i-1}(w) }, \qquad \rf_i(w) = w^{k_i}(w+1)^{-a_i-k_i}e^{tw}\text{ with } \rf_0(w) =1 . 
\end{equation}
Then, we have
\begin{equation}
\prob_Y\left( \bigcap_{\ell=1}^m \{\rx_{k_\ell}(t) \ge a_\ell\}\right) = \det(\rI + \rT),
\end{equation}
where $\rT$ is an integral kernel acting on $L^2(\Sigma_{-1},\rd u/2\pi\ri)$ given by
\begin{equation}
\label{eq:Taround0}
\rT(u,\hu) =  \sum_{\ell=1}^m  \oint \ddbar{v_1}{} \cdots \oint \ddbar{v_\ell}{}  \frac{\rf_\ell(\hu) }{v_\ell -\hu} \cdot  \frac{\ch(v_1,u)}{\left[\prod_{i=1}^{\ell-1}(v_{i}-v_{i+1})\right] \left[\prod_{i=1}^{\ell} \rF_i(v_i) \right] }.
\end{equation}
where $\ch(v,u)$ denotes the characteristic function for TASEP defined in \eqref{eq:chYprobrep}. The integration contours in the above equation are nested small circles around the origin such that $|v_1|>\cdots>|v_m|$.
\end{lm}
\begin{proof}
Recall the equal-time multi-point formula \eqref{eq:equal-timemultinAi}. Splitting $\oint_{|w|>1} = \oint_0 + \oint_{-1}$ and setting $\beta_i=0$, $p_i(v) = (v+1)^{-A_i}e^{tv}$, $q_i(u) = u^{i-1}(u+1)^{-A_i}e^{tu}$, $f_i(v) = (v+1)^{y_i+i}$, $g_i(u) = u^{-i}(u+1)^{y_i+i}$ for all $1 \le i \le N$, and $H(v,u) = -\ch(v,u)$ in Theorem \ref{prop:main}, %and using the relation \eqref{eq:ch_residue}, 
we have
\begin{equation*}
\begin{split}
\rT(u,\hu)
&= -\sum_{j=1}^N \hu^{j-1}(\hu+1)^{-A_j}e^{t\hu} \oint \ddbar{v_1}{}\cdots \oint\ddbar{v_j}{} \frac{\ch(v_1,u)}{\left[\prod_{i=1}^{j-1}(v_i-v_{i+1})\right]\cdot \left[\prod_{i=1}^j v_iP_i(v_i)\right]},
\end{split}
\end{equation*}
where $P_1(v)=p_1(v)$ and $P_i(v)=p_i(v)/p_{i-1}(v)$ for $2\le i\le N$, and the integration contours are nested small circles around the origin satisfying $|v_1|>\cdots>|v_N|$, as in Theorem \ref{prop:main}. Using \eqref{eq:defAi} and the definition of functions $P_i(v)$, we find that
\begin{equation*}
\begin{split}
\rT(u,\hu)
&= -\sum_{\ell=1}^m \sum_{j=k_{\ell-1} + 1}^{k_\ell} \hu^{j-1}(\hu+1)^{-a_\ell-k_\ell}e^{t\hu}\\
& \qquad \cdot \oint \ddbar{v_1}{}\cdots \oint\ddbar{v_j}{} \frac{\ch(v_1,u)e^{-tv_1}}{\left[\prod_{i=1}^{j-1}(v_i-v_{i+1})\right]\cdot \left[\prod_{i=1}^j v_i \cdot \prod_{i=1}^{\ell} h_{i }(v_{k_{i-1}+1})\right] },
\end{split}
\end{equation*}
where we set $a_0=k_0=0$ for convenience, and the function
\begin{equation}
    h_{i }(v):= (v+1)^{a_{i-1}+k_{i-1}-a_i+k_i}\quad \text{for}\quad 1\le i\le m.
\end{equation}
Now we apply the change of variable $j \mapsto j + k_{\ell-1}$:
\begin{equation*}
\begin{split}
\rT(u,\hu)
&= -\sum_{\ell=1}^m \sum_{j=1}^{k_\ell-k_{\ell-1} } \hu^{j+k_{\ell-1} -1}(\hu+1)^{-a_\ell-k_\ell}e^{t\hu} \\
& \qquad \cdot \oint \ddbar{v_1}{}\cdots \oint\ddbar{v_{j+k_{\ell-1} }}{} \frac{\ch(v_1,u)e^{-tv_1}}{\left[\prod_{i=1}^{j+k_{\ell-1} -1}(v_i-v_{i+1})\right]\cdot \left[\prod_{i=1}^{j+k_{\ell-1} } v_i\cdot \prod_{i=1}^{\ell} h_i(v_{k_{i-1}+1})\right] }.
\end{split}
\end{equation*}
Note that only the variables $v_{k_{i-1}+1}$ may have a pole at $-1$ for $1 \le i \le \ell$. The remaining contours can be successively deformed to infinity, collecting residues at $v_{j+1} = v_j$. We obtain
\begin{equation*}
\begin{split}
\rT(u,\hu) &=-\sum_{\ell=1}^m \sum_{j=1}^{k_\ell-k_{\ell-1} } \hu^{j+k_{\ell-1} -1}(\hu+1)^{-a_\ell-k_\ell}e^{t\hu}  \cdot\prod_{i=1}^{\ell}\oint \ddbar{v_{k_{i-1}+1}}{}\\
& \quad \frac{\ch(v_1,u)e^{-tv_1}}{\left[\prod_{i=1}^{\ell-1}(v_{k_{i-1}+1}-v_{k_{i}+1})\right] \left[\prod_{i=1}^{\ell-1 } v^{k_i-k_{i-1}}_{k_{i-1}+1}h_i(v_{k_{i-1}+1}) \right]\left[v_{k_{\ell-1}+1}^j h_{\ell}(v_{k_{\ell-1}+1})\right] }.
\end{split}
\end{equation*}
Summing over $j$, we find that
\begin{equation*}
\begin{split}
\rT(u,\hu) &=-\sum_{\ell=1}^m  \hu^{k_{\ell-1}}(\hu+1)^{-a_\ell-k_\ell}e^{t\hu}  \cdot\prod_{i=1}^{\ell}\oint \ddbar{v_{k_{i-1}+1}}{}  \frac{1}{\hu-v_{k_{\ell-1}+1} } \left(\frac{\hu^{k_{\ell} - k_{\ell-1}}}{v^{k_{\ell} - k_{\ell-1}}_{k_{\ell-1}+1}} - 1\right) \\
& \quad \cdot  \frac{\ch(v_1,u)e^{-tv_1}}{\left[\prod_{i=1}^{\ell-1}(v_{k_{i-1}+1}-v_{k_{i}+1})\right] \left[\prod_{i=1}^{\ell-1 } v^{k_i-k_{i-1}}_{k_{i-1}+1}h_{i}(v_{k_{i-1}+1}) \right]h_\ell(v_{k_{\ell-1}+1}) } \\
&= \sum_{\ell=1}^m  \hu^{k_{\ell}}(\hu+1)^{-a_\ell-k_\ell}e^{t\hu}  \cdot\prod_{i=1}^{\ell}\oint \ddbar{v_{k_{i-1}+1}}{}  \frac{1}{v_{k_{\ell-1}+1} -\hu} \\
& \quad \cdot  \frac{\ch(v_1,u)e^{-tv_1}}{\left[\prod_{i=1}^{\ell-1}(v_{k_{i-1}+1}-v_{k_{i}+1})\right] \left[\prod_{i=1}^{\ell } v^{k_i-k_{i-1}}_{k_{i-1}+1}h_i(v_{k_{i-1}+1}) \right] } \\
&\quad + \sum_{\ell=1}^m  \hu^{k_{\ell-1}}(\hu+1)^{-a_\ell-k_\ell}e^{t\hu}  \cdot\prod_{i=1}^{\ell}\oint \ddbar{v_{k_{i-1}+1}}{}  \frac{1}{\hu-v_{k_{\ell-1}+1} }  \\
& \quad \cdot  \frac{\ch(v_1,u)e^{-tv_1}}{\left[\prod_{i=1}^{\ell-1}(v_{k_{i-1}+1}-v_{k_{i}+1})\right] \left[\prod_{i=1}^{\ell-1 } v^{k_i-k_{i-1}}_{k_{i-1}+1}h_i(v_{k_{i-1}+1}) \right]h_\ell(v_{k_{\ell-1}+1})},
\end{split}
\end{equation*}
where the second term vanishes due to analyticity at $v_{k_{\ell-1}+1}=0$. Renaming the variables $v_{k_{i-1}+1}$ as $v_i$ for $1 \le i \le m$, we obtain \eqref{eq:Taround0}. Hence, the lemma is proved.
\end{proof}
The $v_i$-contours in the formula \eqref{eq:Taround0} for $\rT$ are nested small circles around the origin. We further manipulate the formula for $\rT$ so that the new contours are nested around $-1$; see formula \eqref{eq:Taround-1} below.
\begin{lm}
\label{lm:Taround-1}
Under the assumptions of Lemma \ref{lm:Taround0}, the operator $\rT$ defined in \eqref{eq:Taround0} can be rewritten as follows:
\begin{equation}
\label{eq:Taround-1}
\begin{split}
\rT(u,\hu) &= \sum_{i=1}^{m-1} \oint_0 \ddbar{v}{} \left(\prod_{\ell=i+1}^m \oint\ddbar{u_\ell}{} \right) \frac{\rf_m(\hu)}{\rf_i(v)} \frac{1}{(v-u_{i+1})(\hu-u_m)} \cdot  \frac{\ch(v,u)}{\left[\prod_{\ell=i+1}^{m-1}(u_{\ell+1}-u_{\ell})\right] \left[\prod_{\ell={i+1}}^{m} \rF_\ell(u_\ell) \right] } \\
&\qquad + \oint_0 \ddbar{v}{}  \frac{\rf_{m}(\hu)}{\rf_m(v)} \cdot  \frac{\ch(v,u)}{v -\hu},
\end{split}
\end{equation}
where the contours for $u_2,\ldots,u_m$ are nested small circles around $-1$ such that $|\hu+1|>|u_m+1|>\cdots>|u_2+1|$. The functions $\rf_i$ and $\rF_i$ are defined in \eqref{eq:defrFrf}.
\end{lm}
\begin{proof}
For convenience, we introduce a collection of nested simple closed contours around the points $-1$ and $0$. Recall the regions $\Omega_\LL$ and $\Omega_\RR$ defined in \eqref{eq:regions}.
Suppose $\Sigma_{m+1,\LL}^{\out},\ldots,\Sigma_{1,\LL}^{\out},\Sigma_{1,\LL}^{\inn},\ldots,\Sigma_{m+1,\LL}^\inn$ are $2(m+1)$ nested simple closed contours, from outside to inside, in $\Omega_\LL$ enclosing the point $-1$. Similarly,
$\Sigma_{m,\RR}^{\out},\ldots,\Sigma^{\out}_{1,\RR},\Sigma^{\inn}_{1,\RR},\ldots,\Sigma_{m,\RR}^\inn$ are $2m$ nested simple closed contours, from outside to inside, in $\Omega_\RR$ enclosing the point $0$. These contours are all counterclockwise oriented. Furthermore, assume that $\Sigma_{-1}=\Sigma^\out_{m+1,\LL}$.
For convenience, set $u_{m+1} = \hu$. 
Then we need to show that
\begin{equation}
\rT(u,u_{m+1}) = \sum_{i=1}^{m} \oint_0 \ddbar{v}{} \prod_{\ell=i+1}^m \oint_{\Sigma^\out_{\ell,\LL}}\ddbar{u_\ell}{}  \frac{\rf_{m+1}(u_{m+1})}{\rf_i(v)} \frac{1}{(v-u_{i+1})} \cdot  \frac{\ch(v,u)}{\left[\prod_{\ell=i+1}^{m}(u_{\ell+1}-u_{\ell})\right] \left[\prod_{\ell={i+1}}^{m+1} \rF_\ell(u_\ell) \right]}. 
\end{equation}
Applying Lemma \ref{lem:change_order_contour}, it is enough to show that
\begin{equation}
\label{eq:Tuum+1}
\begin{split}
\rT(u,u_{m+1})
&= \sum_{i=1}^{m}  \oint_0 \ddbar{v}{} \sum_{k=1}^{m-i+1} \sum_{i+1 \le i_1 < \cdots <i_{k} = m +1} \prod_{\ell=1}^{k-1} \oint_{\Sigma^\inn_{i_\ell,\LL}}\ddbar{u_{i_\ell}}{} \\
&\qquad \frac{\rf_{i_k}(u_{i_k})}{\rf_i(v)} \frac{1}{(v-u_{i_1})} \cdot  \frac{\ch(v,u)}{\left[\prod_{\ell=1}^{k-1}(u_{i_{\ell+1}}-u_{i_{\ell}})\right] \left[\prod_{\ell={1}}^{k} \prod_{j=i_{\ell-1}+1}^{i_\ell}\rF_j(u_{i_{\ell}}) \right] },
\end{split}
\end{equation}
where $u_{m+1} \in \Sigma^{\inn}_{m+1, \LL}$, and we set $i_0=i$.

\medskip 

Now consider the expression for $\rT$ in \eqref{eq:Taround0}. In this formula, we may assume that $v_\ell \in \Sigma^\inn_{\ell,\RR}$. We further assume that $\hu  \in \Sigma^{\inn}_{m+1, \LL}$. Applying Lemma \ref{lem:change_order_contour} once again, we find that
\begin{equation*}
\begin{split}
\rT(u,\hu) 
&= \sum_{i=1}^m \prod_{\ell=1}^i  \oint_{\Sigma^\inn_{\ell,\RR}}\ddbar{v_\ell}{}  \frac{\rf_i(\hu) }{v_i -\hu} \cdot  \frac{\ch(v_1,u)}{\left[\prod_{\ell=1}^{i-1}(v_{\ell}-v_{\ell+1})\right] \left[\prod_{\ell=1}^{i} \rF_\ell(v_\ell) \right] } \\
&= \sum_{i=1}^m \sum_{k=1}^i \sum_{1\le i_1 < \cdots < i_k = i} \prod_{\ell=1}^k \oint_{\Sigma^\out_{i_\ell,\RR}} \ddbar{v_{i_\ell}}{}  \frac{\rf_i(\hu) }{v_i - \hu} \cdot  \frac{\ch(v_{i_1},u)}{\left[\prod_{\ell=1}^{k-1}(v_{i_\ell}-v_{i_{\ell+1}})\right] \left[\prod_{\ell=1}^{k} \prod_{j=i_{\ell-1}+1}^{i_\ell} \rF_j(v_{i_\ell}) \right] } \\
&= \sum_{k=1}^m \sum_{1\le i_1 < \cdots < i_k \le m} \prod_{\ell=1}^k \oint_{\Sigma^\out_{i_\ell,\RR}} \ddbar{v_{i_\ell}}{}  \frac{\rf_{i_k}(\hu) }{v_{i_k} - \hu} \cdot  \frac{\ch(v_{i_1},u)}{\left[\prod_{\ell=1}^{k-1}(v_{i_\ell}-v_{i_{\ell+1}})\right] \left[\prod_{\ell=1}^{k} \prod_{j=i_{\ell-1}+1}^{i_\ell} \rF_j(v_{i_\ell}) \right] },
\end{split}
\end{equation*}
where we set $i_0=0$ for convenience.

\medskip 

Recall the definition of $\rF_i$ from \eqref{eq:defrFrf}, along with the assumptions that $k_{i-1} < k_i$ and $a_{i-1}+k_{i-1} \ge a_i+k_i$ for every $2 \le i \le m$. For $\ell \ne 1$, note that the only poles in the variable $v_{i_\ell}$ are at $-1$ and $0$, and that the integrand decays at least as $O(v_{i_\ell}^{-2})$. By deforming the contours ${\Sigma^\out_{i_k,\RR}}, \ldots, {\Sigma^\out_{i_2,\RR}}$ to ${\Sigma^\inn_{i_k,\LL}}, \ldots, {\Sigma^\inn_{i_2,\LL}}$, respectively, in that order, we obtain
\begin{equation*}
\begin{split}
&\rT(u,\hu) \\
&= \sum_{k=1}^m  \sum_{1\le i_1 < \cdots < i_k \le m} (-1)^{k-1} \oint_0 \ddbar{v_{i_1}}{} \prod_{\ell=2}^k \oint_{\Sigma^\inn_{i_\ell,\LL}} \ddbar{v_{i_\ell}}{}   \frac{\rf_{i_k}(\hu) }{v_{i_k} -\hu} \cdot  \frac{\ch(v_{i_1},u)}{\left[\prod_{\ell=1}^{k-1}(v_{i_\ell}-v_{i_{\ell+1}})\right] \left[\prod_{\ell=1}^{k} \prod_{j=i_{\ell-1}+1}^{i_\ell} \rF_j(v_{i_\ell}) \right] } 
\end{split}
\end{equation*}
with $i_0=0$.
We now split the sum over the indices $i_\ell$ as follows:
\begin{equation*}
\begin{split}
\rT(u,\hu) &= \sum_{k=1}^m \sum_{i_1=1}^{m-k+1} \oint_0 \ddbar{v_{i_1}}{} \sum_{i_1+1 \le i_2 < \cdots < i_k \le m}   \prod_{\ell=2}^k \oint_{\Sigma^\inn_{i_\ell,\LL}} \ddbar{v_{i_\ell}}{}   \\
& \quad \frac{\rf_{i_k}(\hu) }{\rf_{i_1}(v_{i_1})} \cdot  \frac{1 }{v_{i_k}-\hu}\frac{\ch(v_{i_1},u)}{\left[\prod_{\ell=1}^{k-1}(v_{i_{\ell+1}}-v_{i_{\ell}})\right] \left[\prod_{\ell=2}^{k} \prod_{j=i_{\ell-1}+1}^{i_\ell} \rF_j(v_{i_\ell}) \right] }. 
\end{split}
\end{equation*}
For convenience, we now set $i_{k+1}=m+1$ and $v_{m+1}=\hu$. Then we have
\begin{equation}
\label{eq:Tuvm+1}
\begin{split}
\rT(u,v_{m+1})
&=  - \sum_{k=1}^m \sum_{i_1=1}^{m-k+1} \oint_0 \ddbar{v_{i_1}}{}  \sum_{i_1+1 \le i_2 < \cdots < i_k<i_{k+1} = m+1}   \prod_{\ell=2}^k \oint_{\Sigma^\inn_{i_\ell,\LL}} \ddbar{v_{i_\ell}}{}   \\
& \quad \frac{\rf_{i_k}(v_{i_{k+1}})}{\rf_{i_1}(v_{i_1})} \cdot  \frac{\ch(v_{i_1},u)}{\left[\prod_{\ell=1}^{k}(v_{i_{\ell+1}}-v_{i_{\ell}})\right] \left[\prod_{\ell=2}^{k} \prod_{j=i_{\ell-1}+1}^{i_\ell} \rF_j(v_{i_\ell}) \right] } \\
&= -\sum_{i_1=1}^{m} \sum_{k=1}^{m-i_1+1}\oint_0 \ddbar{v_{i_1}}{}  \sum_{i_1+1 \le i_2 < \cdots < i_k<i_{k+1} = m+1}   \prod_{\ell=2}^k \oint_{\Sigma^\inn_{i_\ell,\LL}} \ddbar{v_{i_\ell}}{}   \\
& \quad \frac{\rf_{i_{k+1}}(v_{i_{k+1}})}{\rf_{i_1}(v_{i_1})} \cdot  \frac{\ch(v_{i_1},u)}{\left[\prod_{\ell=1}^{k}(v_{i_{\ell+1}}-v_{i_{\ell}})\right] \left[\prod_{\ell=2}^{k+1} \prod_{j=i_{\ell-1}+1}^{i_\ell} \rF_j(v_{i_\ell}) \right] },
\end{split}
\end{equation}
where we have interchanged the sums and used the identity $\prod_{j=i_k+1}^{i_{k+1}}\rF_j(v_{i_{k+1}}) = \rf_{i_{k+1}}(v_{i_{k+1}}) / \rf_{i_{k}}(v_{i_{k+1}})$. Recall $i_1 = i$. Relabeling $i_1$ as $i$ and $v_{i_1}$ as $v$, and relabeling
$(i_2,\ldots,i_{k+1})$ as $(i_1,\ldots,i_k)$ and
$(v_{i_2},\ldots,v_{i_{k+1}})$ as $(u_{i_1},\ldots,u_{i_k})$,
we see that \eqref{eq:Tuvm+1} is equal to \eqref{eq:Tuum+1}.
Hence, the lemma is proved.
\end{proof}

Recall that Lemma \ref{lm:Taround-1} holds when $a_{i-1}+k_{i-1} \ge a_i+k_i$ for all $2 \le i \le m$. We now show that it extends to arbitrary integers $a_1,\ldots,a_m$.

\begin{lm}
Let $m \ge 1$, and $1 \le k_1<\cdots < k_{m-1} < k_m = N$ be $m$ integers. For $a_1,\ldots,a_m \in \intZ$, we have   
\begin{equation} \label{eq:equaltimeTaround-1}
\prob_Y\left( \bigcap_{\ell=1}^m \{\rx_{k_\ell}(t) \ge a_\ell\}\right) = \det\left(\rI + \rT \right)_{L^2(\Sigma_{-1}, \rd u/2\pi \ri)},
\end{equation}
where $\rT$ is given in \eqref{eq:Taround-1}.
\end{lm}
\begin{proof}
Let
$I := \{2 \le i \le m : a_{i-1} + k_{i-1} < a_i + k_i\}$.
The lemma follows by induction on $|I|$. When $|I| = 0$, we have
$a_{i-1} + k_{i-1} \ge a_i + k_i$ for all $2 \le i \le m$. In this case,
\eqref{eq:equaltimeTaround-1} follows from Lemmas \ref{lm:Taround0} and
\ref{lm:Taround-1}.

\medskip
Assume that \eqref{eq:equaltimeTaround-1} holds whenever $|I| = n$, for some $n \ge 0$. We now prove it for the case $|I| = n+1$. Let $2 \le j \le m$ be such that $a_{j-1} + k_{j-1} < a_j + k_j$.

Set $(\hat{a}_1,\ldots, \hat{a}_{m-1}) = (a_1, \ldots, a_{j-2},a_j,\ldots ,a_m)$ and $(\hat{k}_1,\ldots, \hat{k}_{m-1}) = (k_1, \ldots, k_{j-2},k_j,\ldots ,k_m)$. We have
\begin{equation}
\label{eq:distred}
\begin{split}
\prob_Y\left( \bigcap_{\ell=1}^m \{\rx_{k_\ell}(t) \ge a_\ell\}\right) 
= \prob_Y\left( \bigcap_{\ell=1}^{m-1} \{\rx_{\hat{k}_\ell}(t) \ge \hat{a}_\ell\}\right) 
\end{split}
\end{equation}
Recall the formula for $\rF_j$ from \eqref{eq:defrFrf}, and note that the function $1/\rF_j(u)$ is analytic at $u=-1$. We deform the contour for $u_j$ so that it becomes the innermost contour. In doing so, we pick up the residue at $u_j=u_{j-1}$. Then \eqref{eq:Taround-1} becomes
\begin{equation}
\label{eq:Tred}
\begin{split}
\rT(u,\hu) &= \sum_{i=1}^{m-2} \oint_0 \ddbar{v}{} \left(\prod_{\ell=i+1}^{m-1} \oint\ddbar{w_\ell}{} \right) \frac{\hat{\rf}_{m-1}(\hu)}{\hat{\rf}_i(v)} \frac{1}{(v-w_{i+1})(\hu-w_{m-1})} \cdot  \frac{\ch(v,u)}{\left[\prod_{\ell=i+1}^{m-2}(w_{\ell+1}-w_{\ell})\right] \left[\prod_{\ell={i+1}}^{m-1} \hat{\rF}_\ell(w_\ell) \right] } \\
&\qquad + \oint \ddbar{v}{}  \frac{\hat{\rf}_{m-1}(\hu)}{\hat{\rf}_{m-1}(v)} \cdot  \frac{\ch(v,u)}{v -\hu},
\end{split}
\end{equation}
where $(w_2,\ldots,w_{m-1})=(u_2,\ldots,u_{j-2},u_j,\ldots,u_m)$,
$(\hat{\rf}_1,\ldots,\hat{\rf}_{m-1})=(\rf_1,\ldots,\rf_{j-2},\rf_j,\ldots,\rf_m)$,
and $\hat{\rF}_i(w)=\frac{\hat{\rf}_i(w)}{\hat{\rf}_{i-1}(w)}$, with $\hat{\rf}_0(w)=1$. By the induction hypothesis, the expressions in \eqref{eq:distred} and \eqref{eq:Tred} are equal. Hence, the lemma follows.
\end{proof}

%%%%%%%%%%%%%%%%%%%%%%%%%%%%%%%%%%%%%%%%%%%%%%%%%%%%%%%%%%%%
\subsubsection{Equivalence with the path integral formula of \cite{Matetski-Quastel-Remenik21}}

In the next section, we show that the formula $\eqref{eq:equaltimeTaround-1}$ for the equal-time multi-point distribution of TASEP agrees with the path-integral formula of \cite{Matetski-Quastel-Remenik21}. Although the proof is similar to the analogous argument for the KPZ fixed point in \cite{Liao-Liu25}, we include most of the calculations since the calculations of the TASEP formulas are of their own interest. 

\begin{prop}[Theorem 2.6 of \cite{Matetski-Quastel-Remenik21}] \label{prop:mqrpath}
Let $m \ge 1$, and $1 \le k_1 < \cdots < k_{m-1} < k_m = N$ be $m$ integers. For $a_1,\ldots,a_m \in \intZ$, we have 
\begin{equation}
\label{eq:pathintform}
\prob_Y\left( \bigcap_{\ell=1}^m \{\rx_{k_\ell}(t) \ge a_\ell\}\right) = \det\left(\rI - K_t^{(k_m)} + K_t^{(k_m)}Q^{k_1-k_m}\1_{\ge a_1}Q^{k_2-k_1}\1_{\ge a_2} \cdots Q^{k_m-k_{m-1}}\1_{\ge a_m} \right)_{L^2(\intZ)},
\end{equation}
where $K^{(n)}_t$ and $Q^m$ are defined below.
\end{prop}

Recall the geometric random walk $(\hat G_k)_{k\ge 0}$ introduced at the beginning of this subsection (see the discussions before Lemma \ref{lm:Taround0}). The definition of $K^{(n)}_t$ is given by (\cite[Theorem 2.6]{Matetski-Quastel-Remenik21}):
\begin{equation}
\label{eq:defKtn}
K^{(n)}_t= \mathcal{S}^*_{-t,-n} \bar{\mathcal{S}}^{\mathrm{epi}(Y)}_{-t, n},
\end{equation}
where
\begin{equation}
\label{eq:defmcSmcSbar}
\begin{split}
\mathcal{S}_{-t,-n}(x_1,x_2) = \oint_0 \ddbar{v}{} \frac{(1-v)^n e^{t(v - 1/2)} }{2^{x_2-x_1}v^{n+1+x_2-x_1}}, \quad
\bar{\mathcal{S}}_{-t,n}(x_1,x_2) =  \oint_0 \ddbar{v}{}  \frac{(1-v)^{x_2-x_1+n-1} e^{t(v - 1/2)} }{2^{x_1-x_2}v^{n}},
\end{split}
\end{equation}
and
\begin{equation}
\label{eq:SepiY}
\bar{\mathcal{S}}^{\mathrm{epi}(Y)}_{-t, n}(x_1,x_2) = \E_{\hat G_0=x_1} \left[ \bar{\mathcal{S}}_{-t,n-\tau}(\hat G_\tau,x_2) \1_{\tau < n} \right].
\end{equation}
For $m \ge 1$, $Q^m$ is defined as
\begin{equation}
\label{eq:Qm-mdef}
\begin{split}
Q^{m}(x,y) := \frac{1}{2^{x-y}}\binom{x-y-1}{m-1} \1_{x-y\ge m}, \quad
Q^{-m}(x,y) := (-1)^{y-x+m}2^{y-x}\binom{m}{y-x} \1_{y-x\le m}.
\end{split}
\end{equation}
We now express the path-integral kernel in terms of contour integrals. The result is summarized in the following proposition. The proof is similar to that of Proposition 6.4 in \cite{Liao-Liu25}.
\begin{prop}
Assume the same notation and conventions as in Proposition \ref{prop:mqrpath}. Let
\begin{equation*}
\rS := - K_t^{(k_m)} + K_t^{(k_m)}Q^{k_1-k_m}\1_{\ge a_1}Q^{k_2-k_1}\1_{\ge a_2} \cdots Q^{k_m-k_{m-1}}\1_{\ge a_m}.
\end{equation*}
Then we have the following contour integral representation for the kernel of $\rS$:
\begin{equation}
\label{eq:kernelScontour}
\begin{split}
\rS(x_1,x_2) &=  \oint \ddbar{u}{} \oint_0 \ddbar{v}{} \frac{2^{-x_1} (u+1)^{-x_1+a_m-1} }{2^{-x_2}(v+1)^{-x_2+a_m}} \cdot \frac{\rf_m(u)}{\rf_m(v)}\cdot \ch(v,u) \, \1_{x_2<a_m} \\
&\quad + \sum_{i=1}^{m-1}  \oint \ddbar{u}{} \oint_0 \ddbar{v}{} \left(\prod_{\ell={i+1}}^m\oint \ddbar{u_\ell}{}\right)  \frac{2^{-x_1} (u+1)^{-x_1+a_m-1}}{2^{-x_2} (u_m+1)^{-x_2+a_m}} \frac{\rf_m(u)}{ \rf_i(v)} \cdot \ch(v,u) \\
&\qquad \cdot \frac{1}{(v-u_{i+1})\prod_{\ell=i+1}^{m-1}(u_{\ell+1}-u_\ell)} \frac{1}{\prod_{\ell=i+1}^{m}\rF_\ell(u_\ell)} \, \1_{x_2\ge a_m}.
\end{split}
\end{equation}
where the contours for $u,u_2,\ldots,u_m$ are nested small circles around $-1$ such that $|u+1|>|u_m+1|>\cdots>|u_2+1|$. The functions $\rf_i$ and $\rF_i$ are defined in \eqref{eq:defrFrf}.
\end{prop}
\begin{proof}
Note that, writing $\1_{\ge a_1} = 1 - \1_{< a_1}$, we have
\begin{equation*}
\begin{split}
&K_t^{(k_m)}Q^{k_1-k_m}\1_{\ge a_1}Q^{k_2-k_1}\1_{\ge a_2} \cdots Q^{k_m-k_{m-1}}\1_{\ge a_m} \\
&= K_t^{(k_m)}Q^{k_2-k_m}\1_{\ge a_2} \cdots Q^{k_m-k_{m-1}}\1_{\ge a_m} - K_t^{(k_m)}Q^{k_1-k_m}\1_{< a_1}Q^{k_2-k_1}\1_{\ge a_2} \cdots Q^{k_m-k_{m-1}}\1_{\ge a_m},
\end{split}
\end{equation*}
where we used the semigroup property of $Q^{m}$, namely that $Q^m Q^n = Q^{m+n}$, $m,n \in \intZ$. 

Repeating this argument for $K_t^{(k_m)}Q^{k_i-k_m}\1_{\ge a_i} \cdots Q^{k_m-k_{m-1}}\1_{\ge a_m}$ with $i=2,\ldots, m$, we see that
\begin{equation}
\label{eq:2tonexpansion}
\begin{split}
&- K_t^{(k_m)} +K_t^{(k_m)}Q^{k_1-k_m}\1_{\ge a_1}Q^{k_2-k_1}\1_{\ge a_2} \cdots Q^{k_m-k_{m-1}}\1_{\ge a_m} \\
&=  -K_t^{(k_m)} \1_{< a_m} - \sum_{i=1}^{m-1} K_t^{(k_m)}Q^{k_i-k_m}\1_{< a_i}Q^{k_{i+1}-k_i}\1_{\ge a_{i+1}} \cdots Q^{k_m-k_{m-1}}\1_{\ge a_m}.
\end{split}
\end{equation}
From the definition \eqref{eq:defmcSmcSbar}, we have
\begin{equation}
\begin{split}
\mathcal{S}_{-t,-n}(x_1,x_2) = \oint_{-1} \ddbar{u}{} \frac{(-u)^n e^{t(u + 1/2)} }{2^{x_2-x_1}(u+1)^{n+1+x_2-x_1}},\quad
\bar{\mathcal{S}}_{-t,n}(x_1,x_2) =  -\oint_0 \ddbar{v}{}  \frac{(v+1)^{x_2-x_1+n-1} e^{-t(v + 1/2)} }{2^{x_1-x_2}(-v)^{n}}.
\end{split}
\end{equation}
Using \eqref{eq:defKtn} and \eqref{eq:chYprobrep}, we find that
\begin{equation*}
\begin{split}
K^{(N)}_t (x_1,x_2)
&= \sum_{y \in \intZ} \mathcal{S}_{-t,-N}(y,x_1)  \bar{\mathcal{S}}^{\mathrm{epi}(Y)}_{-t, N}(y,x_2) \\
&= -\sum_{y \in \intZ} \oint_{-1} \ddbar{u}{} \frac{(-u)^N e^{t(u + 1/2)} }{2^{x_1-y}(u+1)^{N+1+x_1-y}} \cdot  \E_{\hat G_0=y} \left[ \oint_0 \ddbar{v}{}  \frac{(v+1)^{x_2-\hat G_\tau+N-\tau-1} e^{-t(v + 1/2)} }{2^{\hat G_\tau-x_2}(-v)^{N-\tau}} \1_{\tau<N} \right] \\
&= - \oint_{-1} \ddbar{u}{} \oint_0 \ddbar{v}{} \frac{2^{-x_1}u^N (u+1)^{-x_1-N-1}e^{tu} }{2^{-x_2}v^{N}(v+1)^{-x_2-N} e^{tv}} \cdot \ch(v,u).
\end{split}
\end{equation*}
Note that $k_i< k_m=N$. Using \eqref{eq:Qm-mdef}, we have
\begin{equation}
\label{eq:KtNQexp}
\begin{split}
K^{(N)}_t Q^{k_i-N}(x_1,x_2) &= \sum_{y\in \intZ}K^{(N)}_t(x_1,y)Q^{k_i-N}(y,x_2) \\
&=(-1)^{N-k_i+1}\oint_{-1} \ddbar{u}{} \oint_0 \ddbar{v}{} \frac{2^{-x_1}u^N (u+1)^{-x_1-N-1}e^{tu} }{2^{-x_2}v^{k_i}(v+1)^{-x_2-k_i} e^{tv}} \cdot \ch(v,u) \\
&= (-1)^{N-k_i+1}\oint_{-1} \ddbar{u}{} \oint_0 \ddbar{v}{} \frac{2^{-x_1}(u+1)^{-x_1+a_m-1} \rf_m(u)}{2^{-x_2}(v+1)^{-x_2+a_i}  \rf_i(v)} \cdot \ch(v,u)
\end{split}
\end{equation}
for $1 \le i \le m-1$, where $\rf_i$ is defined in \eqref{eq:defrFrf}. On the other hand, for $m \ge 1$, the kernel $Q^{m}$ in \eqref{eq:Qm-mdef} can be expressed as
\begin{equation*}
\begin{split}
Q^{m}(x,y) 
&= \frac{(-1)^{m}}{2^{x-y}}\oint_{-1} \frac{(u+1)^{m-x+y-1}}{u^m} \ddbar{u}{}.
\end{split}
\end{equation*}
Thus, the convolution $Q^{k_2-k_1}\1_{\ge a_2}Q^{k_3-k_2}$ has the following kernel:
\begin{equation*}
\begin{split}
&Q^{k_2-k_1}\1_{\ge a_2}Q^{k_3-k_2}(x_1,x_2) \\
&= \frac{(-1)^{k_3-k_1}}{2^{x_1-x_2}}\oint_{-1} \ddbar{u_2}{} \sum_{y= a_2}^\infty \oint_{-1}\ddbar{u_3}{}  \frac{(u_2+1)^{k_2-k_1-x_1+y-1}}{u_2^{k_2-k_1}}  \frac{(u_3+1)^{k_3-k_2-y+x_2-1}}{u_3^{k_3-k_2}}  \\
&=  \frac{(-1)^{k_3-k_1}}{2^{x_1-x_2}}\oint_{-1} \ddbar{u_2}{} \oint_{-1}\ddbar{u_3}{}  \frac{(u_2+1)^{k_2-k_1-x_1+a_2-1}}{u_2^{k_2-k_1}}  \frac{(u_3+1)^{k_3-k_2-a_2+x_2}}{u_3^{k_3-k_2}} \frac{1}{u_3-u_2} \\
&= \frac{(-1)^{k_3-k_1}}{2^{x_1-x_2}}\oint_{-1} \ddbar{u_2}{} \oint_{-1}\ddbar{u_3}{}   \frac{(u_2+1)^{a_1-x_1-1}(u_3+1)^{x_2-a_3}}{u_3-u_2} \frac{1}{\rF_2(u_2)\rF_3(u_3)},
\end{split}
\end{equation*}
where $|u_2+1| < |u_3+1|$ and $\rF_i$ is defined in \eqref{eq:defrFrf}. Similarly, for any $1 \le i \le m-1$ we have
\begin{equation}
\label{eq:Q1Qconv}
\begin{split}
&Q^{k_{i+1}-k_i}\1_{\ge a_{i+1}} \cdots 1_{\ge a_{m-1}}Q^{k_m-k_{m-1}}(x_1,x_2) \\
&= \frac{(-1)^{N-k_i}}{2^{x_1-x_2}} \left(\prod_{\ell={i+1}}^m\oint_{-1} \ddbar{u_\ell}{} \right) \frac{(u_{i+1}+1)^{a_i-x_1-1}(u_m+1)^{x_2-a_m}}{\prod_{\ell=i+1}^{m-1}(u_{\ell+1}-u_\ell)} \frac{1}{\prod_{\ell=i+1}^{m}\rF_\ell(u_\ell)},
\end{split}
\end{equation}
where $|u_2+1| < \cdots < |u_m+1|$. Without loss of generality, assume that $|v+1| > |u_m+1|$. Using \eqref{eq:KtNQexp} and \eqref{eq:Q1Qconv}, we have
\begin{equation}
\label{eq:KtNQ1Q}
\begin{split}
&K^{(N)}_t Q^{k_i-N}\1_{< a_{i}}Q^{k_{i+1}-k_i}\1_{\ge a_{i+1}} \cdots 1_{\ge a_{m-1}}Q^{k_m-k_{m-1}}(x_1,x_2) \\
&=- \oint_{-1} \ddbar{u}{} \oint_0 \ddbar{v}{}\sum_{y=-\infty}^{a_i-1} \left(\prod_{\ell={i+1}}^m\oint_{-1} \ddbar{u_\ell}{}\right)  \frac{2^{-x_1} (u+1)^{-x_1+a_m-1} \rf_m(u)}{2^{-y}(v+1)^{-y+a_i} \rf_i(v)} \cdot \ch(v,u) \\
&\qquad \cdot \frac{1}{2^{y-x_2}}\frac{(u_{i+1}+1)^{a_i-y-1}(u_m+1)^{x_2-a_m}}{\prod_{\ell=i+1}^{m-1}(u_{\ell+1}-u_\ell)} \frac{1}{\prod_{\ell=i+1}^{m}\rF_\ell(u_\ell)} \\
&=- \oint_{-1} \ddbar{u}{} \oint_0 \ddbar{v}{} \left(\prod_{\ell={i+1}}^m\oint_{-1} \ddbar{u_\ell}{}\right)  \frac{2^{-x_1} (u+1)^{-x_1+a_m-1}}{2^{-x_2} (u_m+1)^{-x_2+a_m}} \frac{\rf_m(u)}{ \rf_i(v)} \cdot \ch(v,u) \\
&\qquad \cdot \frac{1}{(v-u_{i+1})\prod_{\ell=i+1}^{m-1}(u_{\ell+1}-u_\ell)} \frac{1}{\prod_{\ell=i+1}^{m}\rF_\ell(u_\ell)}.
\end{split}
\end{equation}
Combining \eqref{eq:2tonexpansion} and \eqref{eq:KtNQ1Q}, we obtain the contour representation \eqref{eq:kernelScontour} for $\rS$.
\end{proof}

Finally, we establish the equivalence by appropriately rewriting $\rT$ in \eqref{eq:Taround-1} to match $\rS$ in \eqref{eq:kernelScontour}. We denote the $m$ terms on the right-hand side of \eqref{eq:Taround-1} as $\rT_1+\cdots+\rT_m$. Using
the simple series expansion (for $v$ sufficiently close to the origin, and $|\hat u+1|>|u_m+1|$)
\begin{equation*}
\frac{1}{v-\hu}=\sum_{x=-\infty}^{-1} \frac{(\hu+1)^{-x-1}}{(v+1)^{-x}}, \qquad \frac{1}{\hu-u_m}=\sum_{x=0}^\infty \frac{(\hu+1)^{-x-1}}{(u_m+1)^{-x}},
\end{equation*}
we can write
\begin{equation*}
\rT_i(u,\hu) = \sum_{x \in \intZ} A_i(u,x) B(x,\hu),
\end{equation*}
where
\begin{equation*}
\begin{split}
&A_i(u,x_1) \\
&= 
\frac{\1_{x_1\ge0}}{2^{-x_1}}\oint_0 \ddbar{v}{} \left(\prod_{\ell=i+1}^m \oint\ddbar{u_\ell}{} \right) \frac{1}{\rf_i(v)} \frac{1}{(v-u_{i+1})(u_m+1)^{-x_1}} \cdot  \frac{\ch(v,u)}{\left[\prod_{\ell=i+1}^{m-1}(u_{\ell+1}-u_{\ell})\right] \left[\prod_{\ell={i+1}}^{m} \rF_\ell(u_\ell) \right] }, 
\end{split}
\end{equation*}
for $1\le i\le m-1$, and 
\begin{equation*}
\begin{split}
    A_m(u,x_1)&=\frac{\1_{x_1<0}}{2^{-x_1}}\oint_0 \ddbar{v}{}  \frac{\ch(v,u)}{\rf_m(v)(v+1)^{-x_1}},\\
    B(x_2,\hu)& = 2^{-x_2}\rf_{m}(\hu)(\hu+1)^{-x_2-1}.
\end{split}
\end{equation*}
By Sylvester's identity, we have
\begin{equation*}
\det(\rI + \rT) = \det\left(\rI + \left[\sum_{i=1}^mA_i\right]B\right) = \det\left(\rI + B\left[\sum_{i=1}^mA_i\right]\right) = \det\left(\rI + \sum_{i=1}^m\tilde \rS_i\right),
\end{equation*}
where $\sum_{i=1}^m\tilde \rS_i (x_1,x_2) = \rS(x_1+a_m, x_2+a_m)$. Hence, the equal-time multipoint distribution formula \eqref{eq:equaltimeTaround-1} for TASEP agrees with the path-integral formula \eqref{eq:pathintform} of \cite{Matetski-Quastel-Remenik21}.

\def\cydot{\leavevmode\raise.4ex\hbox{.}}

\end{document}